\documentclass[a4paper,10pt,english,leqno]{amsart}

\usepackage{mathrsfs}
\usepackage{mathtools}
\usepackage{dsfont}

\usepackage{color}
\usepackage[pagebackref=true,colorlinks=false]{hyperref}
\definecolor{darkgreen}{rgb}{0.5,0.25,0}
\definecolor{darkblue}{rgb}{0,0,1}
\definecolor{answerblue}{rgb}{0,0,0.75}

\hypersetup{colorlinks,breaklinks,
            linkcolor=darkblue,urlcolor=darkblue,
            anchorcolor=darkblue,citecolor=darkblue}

\numberwithin{equation}{section}
\allowdisplaybreaks[1]

\usepackage{environ}
\makeatletter
\NewEnviron{Lalign}{\tagsleft@true\begin{align}\BODY\end{align}}
\makeatother

\usepackage[shortcuts,cyremdash]{extdash}

\newtheorem{theorem}{Theorem}[section]
\newtheorem{lemma}[theorem]{Lemma}

\newtheorem{proposition}[theorem]{Proposition}
\theoremstyle{definition}\newtheorem{definition}{Definition}[section]
\theoremstyle{remark}\newtheorem{remark}{Remark}[section]

\newcommand{\R}{\mathbb{R}}

\newcommand{\seq}[1]{\left\{#1\right\}}
\newcommand{\abs}[1]{\left|#1\right|}

\newcommand{\norm}[1]{\left\| #1\right\|}

\newcommand{\totau}{\overset{\tau\downarrow 0}{\longrightarrow}}
\newcommand{\ton}{\overset{n\uparrow \infty}{\longrightarrow}}
\newcommand{\weakn}{\overset{n\uparrow \infty}{\rightharpoonup}}

\usepackage{amssymb}
\usepackage{hyperref}
\usepackage{comment}

\renewcommand{\P}{\mathbb P} 
\newcommand{\EE}{\mathbb E} 

\newcommand{\Lp}[1]{L^{#1}(M,h)} 
\newcommand{\Ltwo}{\Lp{2}} 

\DeclareMathOperator{\Div}{div}
\newcommand{\Divh}{\Div_h}

\newcommand{\Grad}{\mbox{grad}_h\, }
\newcommand{\action}[2]{\left\langle #1,#2 \right\rangle}

\DeclareMathOperator{\supp}{supp}

\DeclareMathOperator{\ind}{Ind}
\DeclareMathOperator{\codim}{codim}

\makeatletter
\@namedef{subjclassname@2020}{%
  \textup{2020} Mathematics Subject Classification}
\makeatother

\title[Stochastic continuity equations]
{Well-posedness of stochastic continuity equations 
on Riemannian manifolds}

\author[L. Galimberti]{Luca Galimberti}
\address[Luca Galimberti]
{\newline Department of Mathematical Sciences
\newline NTNU Norwegian University of Science and Technology
\newline N--7491 Trondheim, Norway} 
\email[]{luca.galimberti@kcl.ac.uk}

\author[K. H. Karlsen]{Kenneth H. Karlsen}
\address[Kenneth Hvistendahl Karlsen]
{\newline Department of mathematics
\newline University of Oslo
\newline P.O. Box 1053,  Blindern
\newline N--0316 Oslo, Norway} 
\email[]{kennethk@math.uio.no}



\thanks{Research Council of Norway supported 
this work through the projects Stochastic 
Conservation Laws (250674) and (in part) 
Waves and Nonlinear Phenomena (250070)}

\date{\today}

\begin{document}

\begin{abstract}
We analyze continuity equations with Stratonovich stochasticity, 
$\partial \rho+ \Div_h \left[ \rho \circ\left(u(t,x)+\sum_{i=1}^N 
a_i(x) \dot W_i(t) \right) \right]=0$, defined 
on a smooth closed Riemannian manifold $M$ with metric $h$. 
The velocity field $u$ is perturbed by Gaussian 
noise terms $\dot W_1(t),\ldots,\dot W_N(t)$ 
driven by smooth spatially dependent vector 
fields $a_1(x),\ldots,a_N(x)$ on $M$. 
The velocity $u$ belongs to $L^1_t W^{1,2}_x$ 
with $\Div_h u$ bounded in $L^p_{t,x}$ for $p>d+2$, where 
$d$ is the dimension of $M$ (we do not assume 
$\Div_h u \in L^\infty_{t,x}$). For carefully 
chosen noise vector fields $a_i$ (and the number $N$ of them), 
we show that the initial-value problem is 
well-posed in the class of weak $L^2$ solutions, 
although the problem can be ill-posed in 
the deterministic case because of concentration 
effects. The proof of this ``regularization by noise" result 
is based on a $L^2$ estimate, which is obtained 
by a duality method, and a weak compactness argument.
\end{abstract}

\maketitle

\vspace{-0.5cm}
{\small \tableofcontents}



\section{Introduction and main results}\label{sec: Introduction}
One of the basic equations in fluid 
dynamics is the continuity equation
$$
\partial_t \rho +\Div \left(u \rho \right) =0 \quad 
\text{in $[0,T]\times \R^d$},
$$
where $u=u(x,t)$ is the velocity field describing the flow 
and $\rho$ is the fluid density. 
It encodes the familiar law of 
conservation of mass. Mathematically speaking, if 
the velocity field $u$ is Lipschitz continuous, 
then the continuity equation (and the related transport 
equation) can be solved explicitly 
by means of the method of characteristics. Unfortunately, 
in realistic applications, the velocity is much rougher than 
Lipschitz, typically $u$ belongs to some spatial Sobolev space 
and one must seek well-posedness of the continuity equation in 
suitable classes of weak solutions. 
Well-posedness of weak solutions follows from 
the theory of renormalized solutions 
\cite{DL89,Ambrosio:2004aa,Lions:NSI,Lions:NSII}, 
assuming that $u\in L^1_tW^{1,1}_x$ (or even $L^1_tBV_x$) 
with $\Div u\in L^\infty_{t,x}$. A key step in this theory 
is to show that a weak solution $\rho$ is 
also a renormalized solution, that is, $S(\rho)$ 
is a weak solution for all ``reasonable" nonlinear 
functions $S:\R\to \R$.  It is the validity 
of this chain rule property that asks 
for $W^{1,1}_x$ (or $BV_x$) regularity of the velocity $u$.
The assumption that $\Div u$ is bounded 
cannot be relaxed (unbounded divergence leads 
to concentration effects).

Recently there has been significant interest in studying 
fluid dynamics equations supplemented with stochastic 
terms. This (renewed) interest is partly motivated 
by the problem of turbulence. Although the basic 
(Navier-Stokes) equations are deterministic, some of 
their solutions exhibit wild random-like behavior, 
with the basic problem of existence and uniqueness of 
smooth solutions being completely open. There is a vague hope 
that ``stochastic perturbations" can render some of the 
models ``well-posed" or  ``better behaved", thereby 
providing  some insight into the onset of turbulence. 
We refer to \cite{Flandoli:2011vn} for a general 
discussion of ``regularization by noise"  phenomena, 
which has been a recurring theme in many recent 
works on stochastic transport and continuity equations of the form
\begin{equation}\label{eq:intro-sTC-Euclid}
	\partial \rho+ \nabla \rho \circ 
	\left (u + a \dot W \right)=0, 
	\quad 
	\partial \rho + 
	\Div \left[ \rho \circ \left (u  + a\dot W\right)\right] =0,
\end{equation}
posed on $\R^d$ with a given initial condition 
$\rho\big |_{t=0}=\rho_0$. Here $W=W(t)$ is 
a Wiener process with noise coefficient $a$ and 
the symbol $\circ$ refers to the Stratonovich stochastic 
differential. It is not our purpose here to review 
the (by now vast) literature on regularization by 
noise (i.e., improvements in regularity, existence, 
uniqueness, stability, etc., induced by noise). 
Instead we emphasize some of the 
papers that develop an analytical (PDE) approach 
\cite{Attanasio:2011fj,Beck:2019,Gess:2018aa,Gess:2019aa,Punshon-Smith:2017aa}, 
related to the one taken in the present paper. 
There is another flexible approach that study 
the stochastic flow associated with the 
SPDE \eqref{eq:intro-sTC-Euclid}, 
relying on regularizing properties of the 
corresponding SDE to supply a flow that 
is more regular than its coefficient $u$, see 
e.g.~\cite{Flandoli-Gubinelli-Priola} for 
the stochastic transport equation and 
\cite{Neves:2015aa,Neves:2016aa} for the 
stochastic continuity equation. A good part 
of the recent literature is motivated by the 
article \cite{Flandoli-Gubinelli-Priola} of 
Flandoli, Gubinelli, and Priola, which in turn built upon an 
earlier work by Davies \cite{Davies}. One of 
the main results in \cite{Flandoli-Gubinelli-Priola} 
is that if $u$ is $x$-H\"older continuous, then 
the initial-value problem for the transport equation 
in \eqref{eq:intro-sTC-Euclid} is well-posed under the 
weak assumption that $\Div u\in L^2$. Most of the works 
just cited assume that the noise coefficient $a$ is constant. 
Well-posedness results for continuity equations 
with $x$-dependent noise coefficients can be 
found in \cite{Punshon-Smith:2017aa} 
(see also \cite{Punshon-Smith:2018aa} and \cite{Holden:2020aa}). 
Subtle regularization by noise results for some 
nonlinear SPDEs can be found \cite{Gess:2018aa,Gess:2019aa}. 
Let us also recall that first order stochastic 
partial differential equations (SPDEs) 
with ``Lipschitz coefficients" have been deeply 
analyzed in Kunita's works \cite{Chow:2015aa,Kunita}.   

In recent years there has been a growing interest in 
analyzing the basic equations of fluid dynamics 
on Riemannian manifolds instead of flat domains, with 
the nonlinearity (curvature) of the domains 
altering the underlying dynamics in nontrivial ways 
(see e.g.~\cite{Amorim:2005aa,Ben-Artzi:2007aa,Rossmanith:2004aa}). 
A Riemannian manifold provides a more general 
framework in which to study fluid dynamics than 
a ``physical surface", with the relevant quantities 
becoming independent of coordinates and a distance function. 
Partial differential equations (PDEs) on manifolds 
arise in many applications, including geophysical 
flows (atmospheric models of fluids constrained to 
flow on the surface of a planet) and general relativity in which 
the Einstein-Euler equations are posed on a manifold with the metric 
being one of the unknowns. Transport equations on manifolds have 
been analyzed in \cite{Dumas:1994aa} and \cite{Fang-Li-Luo}, where 
the DiPerna-Lions theory of weak solutions is 
extended to (some classes of) Riemannian manifolds. 

The mathematical literature on SPDEs on manifolds is at the 
moment scanty, see \cite{Elliott:2012aa,GK0,Gyongy:1993aa,Gyongy:1997aa} 
for equations in which the noise enters the 
equation as an It\^{o} source term. In \cite{GK} we established 
the renormalization property for weak solutions of 
stochastic continuity equations on manifolds, under 
the assumption that the irregular velocity field $u$ 
belongs to $L^1_tW^{1,2}_x$. Corollaries of this 
result included $L^2$ estimates and uniqueness 
(provided $\Div_h u\in L^\infty$). 
The purpose of the present paper is to establish 
the existence and uniqueness of weak $L^2$ 
solutions without the assumption $\Div_h u\in L^\infty$. 

To be more precise, we are given a 
$d$-dimensional ($d\geq 1$) smooth, closed, 
and compact manifold $M$, endowed with a (smooth) Riemannian 
metric $h$. We are interested in the initial-value problem 
for the stochastic continuity equation
\begin{equation}\label{eq:target}
d\rho + \Divh(\rho\, u)dt 
+ \sum_{i=1}^N \Divh ( \rho \,a_i)\circ dW^i(t) 
= 0\, \;\;\; \text{in } \, [0,T]\times M, 
\end{equation}
where $T>0$ denotes a fixed final time, 
$u:[0,T] \times M\to TM$ is a given time-dependent 
irregular vector field on $M$, $a_1,\ldots ,a_N:M\to TM$ 
are suitable smooth vector fields on $M$ (to be fixed later), 
$W^1,\cdots ,W^N$ are independent real-valued Brownian motions, 
and the symbol $\circ$ means that the equation is understood 
in the Stratonovich sense. We recall that for a vector 
field $X$ (locally of the form $X^j \partial_j$), the 
divergence of $X$ is given by $\Divh X=\partial_j X^j 
+ \Gamma^j_{ij} X^i$, where $\Gamma_{ij}^k$ are 
the Christoffel symbols associated with the Levi-Civita 
connection $\nabla$ of the metric $h$ 
(Einstein's summation convention is used throughout the paper).

Roughly speaking, the proof of well-posedness for 
\eqref{eq:target} consists of two main steps. In the first step 
we construct an appropriate noise term that has 
the potential to suppress concentration effects.
Indeed, to remove the assumption $\Div_h u\in L^\infty_{t,x}$, 
we are led to consider a specific noise term linked 
to the geometry of the underlying curved domain $M$, 
implying a structural effect of noise and 
nonlinear domains on improving the well-posedness 
of weak solutions (more on this below).  
Related results on Euclidean domains (with $x$-independent 
noise coefficients) can be found in \cite{Attanasio:2011fj} 
and \cite{Beck:2019} (see also 
\cite{Neves:2015aa,Neves:2016aa,Punshon-Smith:2017aa}).  
In the second step, with help of the noise term, we establish a 
crucial $L^\infty_tL^2_{\omega,x}$ estimate for weak solutions 
that do not depend on $\norm{\Div_h u}_{L^\infty}$. To this end, 
we make use of a duality method, inspired 
by Beck, Flandoli, Gubinelli, and Maurelli \cite{Beck:2019}, 
Gess and Maurelli \cite{Gess:2018aa,Gess:2019aa}, and 
Gess and Smith \cite{Gess:2019aa} (more on this below).

We use the following concept 
of weak solution for \eqref{eq:target} 
(for unexplained notation and background material, 
see Section \ref{sec:background}).

\begin{definition}[weak $L^2$ solution, Stratonovich formulation]
\label{def:L2-weak-sol}
Given $\rho_0\in L^2(M)$, a weak $L^2$ solution of \eqref{eq:target} 
with initial datum $\rho|_{t=0}=\rho_0$ is a function 
$\rho$ that belongs to $L^\infty\left([0,T];L^2(\Omega\times M)\right)$ 
such that $\forall \psi\in C^\infty(M)$ the 
process $(\omega,t)\mapsto\int_M\rho(t)\psi\,dV_h$ has 
a continuous modification which is an 
$\left\{\mathcal{F}_t\right\}_{t\in [0,T]}$-semimartingale 
and for any $t\in [0,T]$ the following equation holds $\P$-a.s.:
\begin{equation*}
\begin{split}
\int_M\rho(t)\psi\,dV_h & = \int_M\rho_0\psi\,dV_h
+\int_0^t\int_M\rho(s)\,u(\psi)\,dV_h\,ds
\\ & \qquad 
+\sum_{i=1}^N\int_0^t\int_M\rho(s)
\, a_i(\psi)\, dV_h\circ dW^i(s).
\end{split}
\end{equation*}
\end{definition}

We have an equivalent concept 
of solution using the It\^o stochastic integral 
and the corresponding SPDE
\begin{equation}\label{eq:target-Ito}
d\rho + \Divh (\rho\, u) \, dt 
+ \sum_{i=1}^N \Divh( \rho \,a_i) \, dW^i(t) 
-\frac{1}{2} \sum_{i=1}^N \Lambda_i(\rho) \, dt= 0,
\end{equation}
where $\Lambda_i$ is a second order differential 
operator linked to the vector field $a_i$, defined by 
$\Lambda_i(\rho):= \Divh \bigl (\Divh (\rho a_i)a_i \bigr)$, 
for $i=1,\ldots,N$. Recall that, for a smooth function $f:M\to \R$ 
and a vector field $X$, we have $X(f)=(X,\Grad f)_h$ 
(which locally becomes $X^j\partial_j f$). Moreover, 
$X\bigl(X(f)\bigr)= (\nabla^2 f)(X,X) + (\nabla_XX)(f)$, 
where $\nabla^2 f$ is the covariant Hessian of $f$ 
and $\nabla_XX$ is the covariant derivative of $X$ in 
the direction $X$. In the It\^{o} SPDE \eqref{eq:target-Ito} 
the operator $\Lambda_i(\cdot)$ is 
the formal adjoint of $a_i\bigl(a_i(\cdot)\bigr)$. 

According to \cite{GK}, the next definition is equivalent to 
Definition \ref{def:L2-weak-sol}.

\begin{definition}[weak $L^2$ solution, It\^o formulation]
\label{def:L2-weak-sol-Ito}
Given $\rho_0\in L^2(M)$, a weak $L^2$ solution 
of \eqref{eq:target} with initial 
datum $\rho|_{t=0}=\rho_0$ is a function $\rho$ that belongs to 
$L^\infty\left([0,T];L^2(\Omega\times M)\right)$ 
such that $\forall \psi\in C^\infty(M)$ the 
stochastic process $(\omega,t)\mapsto\int_M\rho(t)\psi\,dV_h$ 
has a continuous modification which is an 
$\left\{\mathcal{F}_t\right\}_{t\in [0,T]}$-adapted 
process and for any $t\in [0,T]$ 
the following equation holds $\P$-a.s.:
\begin{equation}\label{eq:L2weak-sol-Ito}
\begin{split}
\int_M \rho(t)\psi\, dV_h &= \int_M \rho_0\psi\, dV_h 
+\int_0^t\int_M\rho(s)\,u(\psi)\,dV_h\,ds
\\ & \qquad 
+\sum_{i=1}^N\int_0^t\int_M 
\rho(s) \,a_i(\psi) \, dV_h\, dW^i(s)
\\ & \qquad \qquad 
+\frac12\sum_{i=1}^N \int_0^t \int_M \rho(s) 
\,a_i\bigl(a_i(\psi)\bigr) \, dV_h \,ds.
\end{split}
\end{equation}
\end{definition}

To guarantee that these definitions make sense, we 
need the vector field $u$ to fulfill some 
basic conditions. First, we require 
spatial Sobolev regularity: 
\begin{equation}\label{eq:u-ass-1}
	u\in L^1\left([0,T]; 
	\overrightarrow{W^{1,2}(M)}\right),
\end{equation}
see Section \ref{sec:background} for unexplained notation. 
This means that $u\in L^1\left([0,T];
\overrightarrow{L^2(M)}\right)$, 
which is sufficient to ensure that the mapping
$t\mapsto \int_0^t \int_M \rho(s) u(s)(\psi) \,dV_h\,ds$ 
is absolutely continuous, $\P$-a.s., for any
$\rho\in L^\infty_tL^2_{\omega,x}$ and 
$\psi\in C^\infty(M)$, and hence it is not 
contributing to cross-variations against $W^i$. 
These cross-variations appear when passing from 
Stratonovich to It\^o integrals in the SPDE \eqref{eq:target}.

In addition, we will assume that
\begin{equation}\label{eq:u-ass-2}
u\in L^\infty\left([0,T];
\overrightarrow{L^\infty(M)}\right),
\end{equation}
and, more importantly, that the distributional 
divergence of $u$ satisfies 
\begin{equation}\label{eq:u-ass-3}
\Divh u \in L^p([0,T]\times M), 
\quad \text{for some $p>d+2$}.
\end{equation}

To derive a priori estimates, we need the 
following concept of renormalization (see 
\cite{GK} for details and comments).

\begin{definition}[renormalization, It\^o formulation]
\label{def:renorm}
Let $\rho$ be a weak $L^2$ solution of \eqref{eq:target} 
with initial datum $\rho|_{t=0}=\rho_0\in L^2(M)$. 
We say that $\rho$ is renormalizable if, for 
any $F\in C^2(\R)$ with $F,F',F''$ bounded on $\R$, 
and for any $\psi\in C^\infty(M)$, the stochastic 
process $(\omega,t)\mapsto\int_M F(\rho(t))\psi\,dV_h$ 
has a continuous modification which is 
an $\left\{\mathcal{F}_t\right\}_{t\in [0,T]}$-adapted 
process, and, setting $G_F(\xi):=\xi F'(\xi)-F(\xi)$, 
for $\xi\in \R$, the function $F(\rho)$ satisfy the SPDE
\begin{equation}\label{eq:SPDE-renorm}
	\begin{split}
		& dF(\rho)  + \Divh \bigl(F(\rho) u \bigr)\, dt 
		+ G_F(\rho)\Divh u \, dt
		\\ & \qquad\qquad 
		+\sum_{i=1}^N \Divh\bigl(F(\rho) a_i\bigr)\, dW^i(t)
		+\sum_{i=1}^N G_F(\rho)\Divh a_i \, dW^i(t)
		\\ & \qquad
		= \frac12\sum_{i=1}^N\Lambda_i(F(\rho))\, dt
		-\frac12 \sum_{i=1}^N\Lambda_i(1) G_F(\rho)\,dt 
		\\ & \qquad \qquad 
		+ \frac12\sum_{i=1}^N F''(\rho)
		\,\bigl(\rho\Divh a_i\bigr)^2\, dt
		+\sum_{i=1}^N\Divh \bigl(G_F(\rho)\bar{a}_i \bigr) \,dt,
	\end{split}
\end{equation}
weakly (in $x$), $\P$-a.s., where the first order 
differential operator $\bar{a}_i$ is defined by
$\bar{a}_i :=\left(\Divh a_i\right)a_i$ and 
$\Lambda_i(1) = \Divh \bar{a_i}$, for $i=1,\ldots,N$; 
that is, for all $\psi\in C^\infty(M)$ 
and for any $t\in [0,T]$, the 
following equation holds $\P$-a.s.:
\begin{equation}\label{eq:renorm-weak-x-form}
	\begin{split}
		&\int_M F(\rho(t))\psi \, dV_h 
		= \int_M F(\rho_0)\psi\, dV_h 
		+\int_0^t\int_M F(\rho(s))\,u(\psi)\,dV_h\,ds
		\\ &\quad 
		+\sum_{i=1}^N\int_0^t\int_M 
		F(\rho(s)) \,a_i(\psi) \, dV_h\, dW^i(s)
		+ \frac12 \sum_{i=1}^N \int_0^t \int_M 
		F(\rho(s))\, a_i\bigl(a_i(\psi)\bigr) \, dV_h \,ds
		\\ & \quad 
		-\int_0^t\int_M G_F(\rho(s))
		\Divh u \,\psi\, dV_h\, ds
		-\sum_{i=1}^N\int_0^t\int_MG_F(\rho(s))\Divh a_i 
		\,\psi\,dV_h \, dW^i(s) 
		\\ & \quad 
		-\frac12 \sum_{i=1}^N\int_0^t\int_M 
		\Lambda_i(1)\,G_F(\rho(s))\,\psi\,dV_h\,ds
		\\ & \quad 
		+\frac12\sum_{i=1}^N\int_0^t\int_M 
		F''(\rho(s))\bigl(\rho(s)\Divh a_i\bigr)^2\,\psi\,dV_h\,ds
		\\ & \quad -\sum_{i=1}^N\int_0^t\int_M G_F(\rho(s))
		\bar{a}_i(\psi)\,dV_h \,ds.
	\end{split}
\end{equation}
\end{definition}

\begin{theorem}[renormalization property \cite{GK}]
\label{thm:main-resultGK}
Assume \eqref{eq:u-ass-1} and consider 
a weak $L^2$ solution $\rho$ of \eqref{eq:target} 
with initial datum $\rho_0\in L^2(M)$, according 
to Definition \ref{def:L2-weak-sol-Ito}. 
Then $\rho$ is renormalizable in 
the sense of Definition \ref{def:renorm}. 
\end{theorem}

To prove the $L^2$ estimate mentioned earlier, we 
actually need a version of the weak formulation 
\eqref{eq:renorm-weak-x-form} that uses  
time-dependent test functions. Moreover, we are 
required to insert into that weak formulation 
``non-smooth" test functions. These technical 
aspects of the theory are developed in 
Sections \ref{sec:time-dependent-test} 
and \ref{sec:non-smooth-test}

One can only expect the noise term 
to improve the well-posedness situation for \eqref{eq:target} 
if the resulting second order differential 
operator $\sum_i a_i(a_i)$ appearing 
in \eqref{eq:L2weak-sol-Ito} is 
non-degenerate (uniformly elliptic). 
The required non-degeneracy is not guaranteed. 
The reason is a geometric one that is tied to 
the nonlinearity of the domain.  Indeed, given an 
arbitrary $d$-dimensional smooth manifold $M$, it is 
not possible to find a global frame for it, that is, $d$ smooth 
vector fields $E_1,\ldots ,E_d$ that constitute 
a basis for $T_xM$ for all $x\in M$. 
Manifolds that exhibit this property are called parallelizable. 
Examples of parallelizable manifolds are 
Lie groups (e.g.~$\R^d$, $\mathbb{T}^d$) 
and the sphere $\mathbb{S}^d$ with $d\in \{1,3,7\}$.
We refer to Section \ref{sec: ellipticity} 
for further details, and a proof 
of the following simple but useful fact:

\begin{lemma}[non-degenerate second order operator]
\label{lem:ellipticity}
There exist $N=N(M)$ smooth vector 
fields $a_1,\ldots,a_N$ on $M$ such that 
the following identity holds
$$
\frac{1}{2}\sum_{i=1}^N a_i\bigl(a_i(\psi)\bigr) 
= \Delta_h\psi - \frac{1}{2}
\sum_{i=1}^N \bar{a}_i(\psi),
\quad \forall \psi\in C^2(M),
$$
where $\Delta_h$ is the Laplace-Beltrami 
operator of $(M,h)$ and $\bar{a}_1,\ldots,\bar{a}_N$ 
are first order differential operators: 
$\bar{a}_i := (\Divh a_i) \,a_i$, for $i=1,\ldots,N$.
\end{lemma}

It is now clear that with the specific 
vector fields $a_1,\ldots,a_N$ constructed in 
Lemma \ref {lem:ellipticity}, the resulting 
second order operator $\frac{1}{2}\sum_{i=1}^N a_i(a_i(\cdot))$ 
in \eqref{eq:L2weak-sol-Ito} becomes uniformly elliptic. 
The main result of this paper, which shows how 
the use of noise can avoid concentration 
in the density $\rho$, is the 
following theorem: 

\begin{theorem}[well-posedness]\label{thm:main-result}
Suppose conditions \eqref{eq:u-ass-1}, \eqref{eq:u-ass-2}, 
and \eqref{eq:u-ass-3} hold. Let the vector 
fields $a_1,\ldots ,a_N$ be given by Lemma \ref{lem:ellipticity}. 
Then there exists a unique weak solution of \eqref{eq:target} 
with initial datum $\rho|_{t=0}=\rho_0\in L^2(M)$.
\end{theorem}

As far as we know, this theorem provides the first result on 
regularization by noise on a manifold (we are not aware of 
any such result even for ODEs). The proof consists 
of several steps. The first one establishes the well-posedness 
of strong solutions to \eqref{eq:target} with smooth 
data ($u,\rho_0$), which is the topic of 
Section \ref{sec:strong-sol}. Here the basic strategy is, 
with the help of a smooth partition of unity subordinate 
to a finite atlas, to solve localized versions 
of \eqref{eq:target} ``pulled back'' to $\R^d$, 
relying on Kunita's existence and uniqueness theory 
for SPDEs on Euclidean domains \cite{Kunita82,Kunita}. 
We ``glue" the localized solutions together on $M$, 
obtaining in this way a global solution. 
The gluing procedure is well-defined, because if two 
coordinate patches intersect, then their 
corresponding solutions must agree on the 
intersection, in view of the uniqueness result 
that is available on $\R^d$ (with $u,\rho_0$ smooth). 
In Section \ref{sec:L2-est}, we derive an $L^2$ estimate for 
general weak solutions $\rho$ (with non-smooth $u,\rho_0$):
\begin{equation}\label{eq:intro-L2est}
\EE \int_M \abs{\rho(t,x)}^2\, dx 
\le C\int_M \abs{\rho_0(x)}^2\, dx,
\quad t \in [0,T],
\end{equation}
where the constant $C$ depends on 
$\norm{\Divh u}_{L^p_{t,x}}$, see \eqref{eq:u-ass-3}, 
but not $\norm{\Divh u}_{L^\infty}$. 
The derivation of this estimate is based on 
\eqref{eq:SPDE-renorm}, and a duality 
argument in which we construct a specific 
(deterministic) test function $\phi(t,x)$ 
that can ``absorb" the bad $\Divh u$ term 
in \eqref{eq:SPDE-renorm}. This function solves 
the terminal-value problem
$$
\partial_t\phi(t) + \Delta_h\phi(t) 
- b(t,x) \phi= 0 \;\; \mbox{on } [0,t_0]\times M,
\qquad \phi(t_0,x) = 1 \;\; \mbox{on } M,
$$
where $t_0\in [0,T]$, $\Delta_h$ is the Laplace-Beltrami 
operator, and $b=b(t,x)\le 0$ is an appropriately chosen 
irregular function ($b\in L^p$ with $p>d+2$). 
Using Fredholm theory and embedding theorems in 
anisotropic Sobolev spaces $W^{1,2,p}_{t,x}$ \cite{Besov}, 
we prove that this problem admits a unique 
solution $\phi \in W^{1,2,p}_{t,x}$ 
that satisfies
\begin{equation}\label{eq:intro-phi-est}
\norm{\left(\phi,\nabla \phi\right)}_{L^\infty_{t,x}}
\le C\left(p,d,T,M\right)\norm{b}_{L^p_{t,x}},
\end{equation}
where $\nabla$ denotes the covariant derivative. 
Using this $\phi$ as test function in the 
time-space weak formulation of \eqref{eq:SPDE-renorm}, 
along with the estimates \eqref{eq:intro-phi-est}, 
we arrive at the $L^2$ estimate \eqref{eq:intro-L2est} 
via Gr{\"o}nwall's inequality. In the final 
step (Section \ref{sec:proof-main-result}), 
we replace the irregular vector field $u$ and the 
initial function $\rho_0\in L^2$ by appropriate 
smooth approximations $u_\tau$ and $\rho_{0,\tau}$, 
respectively, where $\tau>0$ is the approximation 
parameter, and solve for each $\tau>0$ the 
corresponding SPDE with smooth data 
$\left(u_\tau, \rho_{0,\tau}\right)$, giving raise to a 
sequence $\seq{\rho_\tau}_{\tau>0}$ of approximate solutions. 
In view of \eqref{eq:intro-L2est}, we have 
an $L^2$ bound on $\rho_\tau$ that is 
independent of $\tau$, which is enough to arrive at 
the existence of a weak solution to \eqref{eq:target} by way of 
a compactness argument. Uniqueness is an immediate 
consequence of \eqref{eq:intro-L2est}.

Before ending this (long) section, let us briefly 
discuss the nontrivial matter of regularizing functions 
and vector fields on manifolds. 
In the Euclidean case one uses mollification. 
Mollification possesses many fitting 
properties (e.g.~it commutes with differential operators) 
that are not easy to engineer if the function in 
question is defined on a manifold. 
Indeed, on a Riemannian 
manifold, there are a number of smoothing devices 
currently being used, including partition of unity 
combined with Euclidean convolution in local 
charts (see~e.g.~\cite{Dumas:1994aa,GK0,GK}), 
Riemannian convolution smoothing 
\cite{Greene:1979aa}, and the heat semigroup 
method (see~e.g.~\cite{Fang-Li-Luo,GK0}), 
where the last two are better at 
preserving geometric properties. 
In this paper, for smoothing of the data $\rho_0$ 
(function) and $u$ (vector field), we employ standard 
mollifcation in time and convolution with 
the heat semigroup in the spatial variables, 
where the heat semigroup approach is applied to 
functions as well as vector fields (the latter 
via 1-forms and the de Rham-Hodge semigroup), see 
Section \ref{sec:appendix} for details.


\section{Background material}\label{sec:background}

In an attempt to make the paper 
more self-contained and fix relevant notation, we 
briefly review some basic aspects of differential 
geometry and stochastic analysis. For unexplained 
terminology and rudimentary results concerning the target 
equation \eqref{eq:target}, we refer to \cite{GK}.


\subsection{Geometric framework}\label{sec:geometric framework}
We refer to \cite{Aubin,deRham,LeeRiemann} for background 
material on differential geometry and analysis on manifolds. 
Fix a closed, compact, connected and oriented $d$-dimensional 
smooth Riemannian manifold $(M,h)$. The metric $h$ 
is a smooth positive-definite 2-covariant tensor field, 
which determines for every $x\in M$ an inner product $h_x$ 
on $T_xM$. Here $T_xM$ denotes the tangent 
space at $x$, and by $TM=\coprod_{x\in M} T_xM$ 
we denote the tangent bundle. For two arbitrary 
vectors $X_1,X_2\in T_xM$, we will henceforth 
write $h_x(X_1,X_2)=:\left(X_1,X_2\right)_{h_x}$ 
or even $\left(X_1,X_2\right)_h$ if the context is clear. 
We set $\abs{X}_h:=\left( X, X\right)_h^{1/2}$. 
Recall that, in local coordinates $x=(x^i)$, the 
partial derivatives $\partial_i:=\frac{\partial}{\partial x^i}$ 
form a basis for $T_xM$, while the differential forms $dx^i$ 
determine a basis for the cotangent space $T_x^\ast M$. 
Therefore, in local coordinates, $h$ reads 
$$
h = h_{ij}\, dx^i dx^j,  \quad 
h_{ij}=\left(\partial_i,\partial_j\right)_h. 
$$
We will denote by $(h^{ij})$ the 
inverse of the matrix $(h_{ij})$.

We denote by $dV_h$ the Riemannian 
density associated to $h$, which 
in local coordinates takes the form
$$
dV_h = \abs{h}^{1/2}\, dx^1\cdots dx^d, 
$$
where $\abs{h}$ is the determinant of $h$.  
Throughout the paper, we will assume 
for convenience that 
$$
\mathrm{Vol}(M,h):=\int_MdV_h=1. 
$$
For $p\in [1,\infty]$, we denote 
by $L^p(M)$ the usual Lebesgue spaces on $(M,h)$. 
In local coordinates, the gradient 
of a function $f:M\to \R$ is the vector 
field given by the following expression
$$
\Grad f := h^{ij} \partial_if\,\partial_j. 
$$

The symbol $\nabla$ refers to the Levi-Civita 
connection of $h$, namely the unique linear connection 
on $M$ that is compatible with $h$ and is symmetric. 
The Christoffel symbols associated 
to $\nabla$ are given by
$$
\Gamma^k_{ij} = \frac12 h^{kl} \left( \partial_ih_{jl}
+\partial_jh_{il}-\partial_lh_{ij} \right). 
$$
In particular, the covariant derivative of a vector field 
$X=X^\alpha\partial_\alpha$ is the $(1,1)$-tensor field which 
in local coordinates reads
$$
(\nabla X)_j^\alpha  := 
\partial_j X^\alpha + \Gamma^\alpha_{kj}X^k.
$$
The divergence of a vector field $X=X^j\partial_j$ 
is the function defined by
$$
\Divh X := \partial_jX^j + \Gamma^j_{kj}X^k. 
$$

For any vector field $X$ and $f\in C^1(M)$, we have 
$X(f)= (X,\Grad f)_h$, which locally 
takes the form $X^j \partial_jf$. We recall that 
for a (smooth) vector field $X$, the following 
integration by parts formula holds:
$$
\int_M X(f) \,dV_h =\int_M \left(\Grad f,X\right)_h
\,dV_h 
=-\int_M f \,\Divh X \, dV_h,
$$
recalling that $M$ is closed (so all 
functions are compactly supported).

Given a smooth vector field $X$ on $M$, 
we consider the  norm
$$
\norm{X}_{\overrightarrow{L^p(M)}}^p
:= 
\begin{cases}
\displaystyle \int_M \abs{X}^p_h \,dV_h,
& p\in [1,\infty), \\ 
\norm{\abs{X}_h}_{L^\infty(M)}.
& p=\infty.
\end{cases}
$$
The closure of the space of smooth 
vector fields on $M$ with respect 
to the norm $\norm{\cdot}_{\overrightarrow{L^p(M)}}$ is 
denoted by $\overrightarrow{L^p(M)}$. 
We define the Sobolev space 
$\overrightarrow{W^{1,p}(M)}$ in a similar fashion. 
Indeed, consider the norm
$$
\norm{X}_{\overrightarrow{W^{1,p}(M)}}^p
:= \begin{cases}
\displaystyle 
\int_M \bigl(\abs{X}^p_h 
+ \abs{\nabla X}^p_h\bigr) \,dV_h, 
& \text{if $p\in [1,\infty$)},
\\  
\norm{\, \abs{X}_h + \abs{\nabla X}_h\, }_{L^\infty(M)},
& \text{if $p=\infty$}
\end{cases}
$$
where, locally, $\abs{\nabla X}_h^2
=(\nabla X)^i_j \, h_{ik}h^{jm} \, (\nabla X)^k_m$. 
The closure of the space of smooth vector fields with respect  
to this norm is $\overrightarrow{W^{1,p}(M)}$. 
For more operative definitions, $\overrightarrow{L^p(M)}$ 
and $\overrightarrow{W^{1,p}(M)}$ can be seen 
as the spaces of vector fields whose 
components in any arbitrary chart 
belong to the corresponding Euclidean space.

We will make essential use of the anisotropic Sobolev space 
$W^{1,2,p}([0,T]\times M)$, with $p\in [1,\infty)$ and 
$T>0$ finite. This space is defined as the completion of 
$C^\infty([0,T]\times M)$ under the norm
\begin{equation}\label{eq:anisotropic-norm}
    \norm{w}_{W^{1,2,p}([0,T]\times M)}:=
    \sum_{\substack{j,k\ge 0 \\ 2j+k\leq 2}}
    \left[\iint_{[0,T]\times M}
    \abs{\partial_t^j 
    \nabla^k w}^p_h\,dt \,dV_h\right]^{1/p},
\end{equation}
where $\nabla^k w$ denotes the $k$th covariant 
derivative of the function $w$. 
We have the following important embedding result (see
Section \ref{sec:appendix} for a proof):
\begin{proposition}\label{prop:anisotropic-embed}
Suppose $p> d+2$. Then 
$$
W^{1,2,p}([0,T]\times M) \subset \subset 
C^{0,1-\frac{1+d}{p}}([0,T]\times M);
$$
the first-order $x$-derivatives of a 
function $w=w(t,x)\in W^{1,2,p}([0,T]\times M)$ are 
H\"older continuous with exponent 
$1-\frac{1+d}{p}$, such that
$$
\norm{w}_{C^0([0,T]\times M)} 
+\norm{\nabla w}_{C^0([0,T]\times M)}
\leq C\norm{w}_{W^{1,2,p}([0,T]\times M)},
$$
for some constant $C=C(p,d,M)$.
\end{proposition}

Finally, we introduce the following 
second order differential operators 
associated with the vector fields $a_1,\ldots,a_N$:
\begin{equation}\label{eq:def-Lambda-i}
	\Lambda_i(\psi):= 
	\Divh \bigl(\Divh(\psi a_i)a_i\bigr), 
	\quad \psi\in C^2(M), 
	\quad i=1,\ldots,N.
\end{equation}

It is not difficult to see that the 
adjoint of $\Lambda_i(\cdot)$ is 
$a_i\bigl(a_i(\cdot)\bigr)$:
\begin{align*}
\int_M \Lambda_i(\psi)\phi\,dV_h 
&= \int_M \psi \, a_i\bigl(a_i(\phi) \bigr)
\,dV_h 
\\ & =\int_M \psi\,
\Bigl((\nabla^2\phi)(a_i,a_i) 
+(\nabla_{a_i}a_i)(\psi)\Bigr)\,dV_h,
\quad \forall \psi,\phi\in C^2(M),
\end{align*}
see \cite{GK} for further details.

\subsection{Stochastic framework}
We use the books \cite{Protter,Revuz:1999aa} as general 
references on the topic of stochastic analysis. 
From beginning to end, we fix a complete 
probability space $(\Omega, \mathcal{F}, \P)$ and 
a complete right-continuous filtration 
$\seq{\mathcal{F}_t}_{t\in [0,T]}$. 
Without loss of generality, we assume that the $\sigma$-algebra 
$\mathcal{F}$ is countably generated.  
Let $W=\seq{W_i}_{i=1}^N$ be a finite sequence 
of independent one-dimensional 
Brownian motions adapted to the filtration
$\seq{\mathcal{F}_t}_{t\in [0,T]}$. We refer to 
$\bigl(\Omega,\mathcal{F},
\seq{\mathcal{F}_t}_{t\in [0,T]},\P, W\bigr)$ 
as a (Brownian) \textit{stochastic basis}. 

Consider two real-valued stochastic processes $Y,\tilde Y$. 
We call $\tilde Y$ a \textit{modification} of $Y$ 
if, for each $t\in [0,T]$, $\P\bigl(\bigl\{\omega\in \Omega:Y(\omega,t)
=\tilde Y(\omega,t)\bigr\}\bigr)=1$.  
It is important to pick good modifications of stochastic processes. 
Right (or left) continuous modifications are 
often used (they are known to exist 
for rather general processes), since any two such modifications 
of the same process are indistinguishable (with 
probability one they have the same sample paths). 
Besides, they necessarily have left-limits everywhere. 
Right-continuous processes with left-limits 
are referred to as \textit{c{\`a}dl{\`a}g}.

An $\left\{\mathcal{F}_t\right\}_{t\in [0,T]}$-adapted, 
c{\`a}dl{\`a}g process $Y$ is an 
$\left\{\mathcal{F}_t\right\}_{t\in [0,T]}$-semimartingale 
if there exist processes $F,M$ with $F_0=M_0=0$ such that
$$
Y_t = Y_0 + F_t+ M_t,
$$
where $F$ is a finite variation process 
and $M$ is a local martingale. 
In this paper we will only be concerned 
with \textit{continuous} semimartingales. 
The quantifier ``local'' refers to 
the existence of a sequence $\seq{\tau_n}_{n\ge1}$ 
of stopping times increasing to infinity such 
that the stopped processes 
$\mathbf{1}_{\seq{\tau_n>0}} M_{t\wedge \tau_n}$ 
are martingales. 

Given two continuous semimartingales $Y$ and $Z$, we can define 
the Fisk-Stratonovich integral of $Y$ with respect to $Z$ by
$$
\int_0^tY(s)\circ dZ(s) = \int_0^tY(s) \, dZ(s) 
+ \frac12 \action{Y}{Z}_t,
$$
where $\int_0^tY(s) dZ(s)$ is the It\^o integral of $Y$ 
with respect to $Z$ and $\action{Y}{Z}$ 
denotes the quadratic cross-variation process of $Y$  
and $Z$.  Let us recall It\^o's formula for 
a continuous semimartingale $Y$. 
Let $F\in C^2(\R)$. Then $F(Y)$ is again 
a continuous semimartingale and 
the following chain rule formula holds:
$$
F(Y(t))-F(Y(0))= \int_{0}^t F'(Y(s))dY(s) 
+ \frac12\int_0^t F''(Y(s)) \, d\action{Y}{Y}_s.
$$

Martingale inequalities are generally important 
for several reasons. For us they will be used to 
bound It\^{o} stochastic integrals in terms 
of their quadratic variation (which is easy to compute). 
One of the most important martingale inequalities 
is the Burkholder-Davis-Gundy inequality. 
Let $Y=\seq{Y_t}_{t\in [0,T]}$ be a 
continuous local martingale with $Y_0=0$. 
Then, for any stopping time $\tau \leq T$,
\begin{equation*}
\EE \left(\sup_{t\in[0,\tau]}\abs{Y_t}\right)^p 
\leq C_p\, \EE 
\sqrt[\leftroot{-2}\uproot{-1}p]{\action{Y}{Y}_\tau},
\qquad p\in (0,\infty),
\end{equation*}
where $C_p$ is a universal constant.


\section{Smooth data and strong solutions}\label{sec:strong-sol}


\subsection{Strong solution}
We are going to construct strong solutions 
to \eqref{eq:target} when the data $\rho_0,u$ are smooth. 
More precisely, throughout this section, we will 
assume $\rho_0\in C^\infty(M)$ and that 
$u:[0,\infty)\times M\to TM$ is a vector field on $M$ 
that is smooth in both variables. The strategy we 
employ is the following one: firstly we solve 
a local version of \eqref{eq:target} ``pulled back'' on $\R^d$, 
applying the ``Euclidean" existence and uniqueness 
theory developed in \cite{Kunita82}. In a second 
step we glue these solutions all together on $M$, 
obtaining a global solution. The gluing procedure is 
well-posed because there is a uniqueness result on $\R^d$ 
for smooth data $(\rho_0,u)$.

Fixing a point $p\in M$, we may find an open 
neighborhood $\mathcal{U}(p)\subset M$ of $p$ and 
coordinates $\gamma_p:\mathcal{U}(p)\to \R^d$ such 
that $\gamma_p(\mathcal{U}(p))=\R^d$. By compactness of $M$, 
there is a finite atlas $\mathcal{A}$ with 
these properties, namely there exist $p_1,\ldots,p_K$ 
such that $M=\bigcup_{l=1}^K\mathcal{U}(p_l)$ 
and $\gamma_{p_l}\bigl(\mathcal{U}(p_l)\bigr)=\R^d$.

\begin{remark}\label{rem:bounded-deriv}
To construct these coordinates, one is usually led 
to use that the ball $B_1(0)\subset\R^d$ is diffeomorphic 
to the whole space, for instance via the map
$$
\Phi:\R^d\to B_1(0), \quad 
z\mapsto \frac{z}{\sqrt{1+|z|^2}}.
$$
If we now have a $C^1$ function $f:B_1(0)\to \R$ 
with bounded derivatives, then it is 
straightforward to check that $f\circ\Phi$ 
has bounded derivatives as well.
\end{remark}

Fix a point $p_l$. In the 
coordinates given by $\gamma_{p_l}$, \eqref{eq:target} 
looks like 
\begin{equation*}
\begin{split}
& d\rho^l + \left[ \rho^l\Divh u 
 + \partial_j\rho^l u^j\right] \, dt 
\\ & \qquad \qquad
+ \sum_{i=1}^N \left[ \partial_j\rho^la_i^j 
+ \rho^l\Divh a_i\right]\circ dW^i_t = 0 
\quad \text{on $[0,T]\times \R^d$}, 
\\ & \rho^l(0)= \rho_0\circ \gamma_{p_l}^{-1}
\quad \mbox{on $\R^d$}.
\end{split}
\end{equation*}
Observe that the coefficients satisfy 
the hypotheses on pages 264 and 267 in 
\cite{Kunita82}. In particular, the $z$-derivatives 
are bounded, in view of Remark \ref{rem:bounded-deriv}. 
(In Kunita's notation we have
\begin{align*}
& Q_0(t,x,v)=v\Divh u, 
\quad 
Q_j(t,x,v)=v \Divh a_j, 
\quad 
j=1,\ldots , N
\\ &
P_0^r = u^r, 
\quad 
P^r_j = a^r_j, 
\quad 
r=1,\ldots,d, 
\quad 
j=1,\ldots,N
\\ &
Q^{(1)}_0 = \Divh u, 
\quad Q^{(0)}_0=0,
\\ &
Q^{(1)}_j = \Divh a_j, 
\quad 
Q^{(0)}_j=0, 
\quad 
j=1,\ldots , N.)
\end{align*}
Therefore, we may apply \cite[Theorem 4.2]{Kunita82} 
to obtain a unique strong solution 
which we call $\rho^l$ (for the definition 
of strong solution, see \cite[p.~255]{Kunita82}). 
Let us ``lift'' $\rho^l$ on $M$, 
via $\gamma_{p_l}$, namely, for 
$t\in [0,T]$ define
$$
\hat{\rho}^l(t,x):=
\begin{cases}
\rho^l\bigl(t,\gamma_{p_l}(x)\bigr), 
& x\in \mathcal{U}(p_l),
\\ 
0, 
& x\notin \mathcal{U}(p_l).
\end{cases}
$$
We repeat this procedure for all $p_l$, 
thereby obtaining $\hat{\rho}^l$ 
for $l\in\seq{1,\ldots, K}$. 

Suppose that $\mathcal{U}(p_l)
\cap\mathcal{U}(p_{l'})\neq \emptyset$, 
for some $l\neq l'$. Fix 
$q\in \mathcal{U}(p_l)\cap\mathcal{U}(p_{l'})$. 
Arguing as above, we may find coordinates 
$\eta_q:\mathcal{V}(q)\to \R^d$ such that 
$\mathcal{V}(q)$ is an open neighborhood of $q$ with 
$\mathcal{V}(q)\subset \mathcal{U}(p_l)
\cap\mathcal{U}(p_{l'})$, and 
$\eta_q\bigl(\mathcal{V}(q)\bigr)=\R^d$. 
Once again, we can find a unique 
strong solution $\rho_q$, which we 
lift on $M$: for $t\in [0,T]$ define
$$
\hat{\rho}_q(t,x)
:=
\begin{cases}
\rho_q\bigl(t,\eta_{q}(x)\bigr), 
& x\in \mathcal{V}(q),
\\
0, 
& x\notin \mathcal{V}(q).
\end{cases}
$$

We now restrict $\hat{\rho}^l(t,\cdot)$ 
on $\mathcal{V}(q)$. Trivially, the 
restriction satisfies \eqref{eq:target} 
on $\mathcal{V}(q)$. This is a geometric equation 
(and thus coordinate-independent), which implies that 
the restriction of $\hat{\rho}^l(t,\cdot)$ 
to $\mathcal{V}(q)$ must satisfy \eqref{eq:target} 
when written in the coordinates given by $\eta_q$. 
By uniqueness in $\R^d$ (of strong solutions), 
we must have $\hat{\rho}^l\bigl(\cdot,\eta_q^{-1}(\cdot)\bigr)
=\rho_q(\cdot,\cdot)$ on $[0,\infty)\times \R^d$, 
and thus $\hat{\rho}^l(t,x)=\hat{\rho}_q(t,x)$ 
for all $t\in [0,T]$ and $x\in\mathcal{V}(q)$. 
By symmetry, we infer
$$
\hat{\rho}^l(t,x)= \hat{\rho}^{l'}(t,x), 
\quad \text{for $(t,x)\in [0,T]
\times \mathcal{V}(q)$}.
$$
Repeating the whole procedure for all 
$q\in \mathcal{U}(p_l)\cap\mathcal{U}(p_{l'})$, 
we conclude that
$$
\hat{\rho}^l(t,x)= \hat{\rho}^{l'}(t,x), 
\quad \text{for $(t,x)\in [0,T]\times 
\bigl(\mathcal{U}(p_l)\cap\mathcal{U}(p_{l'})\bigr)$}.
$$
In view of these compatibility conditions, 
we may unambiguously define 
\begin{equation}\label{eq:strong-sol-defining-eqn}
\rho(t,x):= \hat{\rho}^l(t,x), 
\quad (t,x)\in [0,T]\times M,
\end{equation}
where $l$ is an index in $\seq{1,\ldots, K}$ 
such that $x\in\mathcal{U}(p_l)$. 

We have thus arrived at

\begin{lemma}[strong solution, smooth data]
\label{lem:exist-strong-sol}
The function $\rho$ given by \eqref{eq:strong-sol-defining-eqn} 
is the unique strong solution of \eqref{eq:target} 
with initial datum $\rho_0\in C^\infty(M)$ and 
smooth vector field $u:[0,\infty)\times M\to TM$. 
Moreover, $\rho$ is a $C^\infty$ semimartingale.
\end{lemma}


\subsection{Elementary $L^p$ bound}
Let $\rho$ be the solution constructed above. 
We observe that, in view of the 
results in \cite{Kunita82}, locally in 
the coordinates induced by $\gamma_{p_l}$ 
on $\mathcal{U}(p_l)$, we have the following explicit 
expression for $\rho$:
\begin{equation}\label{eq:expression-strong-sol}
\begin{split}
\rho & \bigl(t,\gamma_{p_l}^{-1} (z)\bigr) 
= \rho^l(t,z)
\\ & \quad 
=\exp\left(\int_0^t\Divh u(s,\xi_s(y))\, ds
+\sum_{i=1}^N\int_0^t\Divh a_i(y)
\circ dW_s^i\right)\Bigg|_{y=\xi_t^{-1}(z)}
\\ & \quad \quad \quad \quad \quad \quad 
\times \rho_0\bigl(\gamma_{p_l}^{-1} \circ \xi^{-1}_t(z)\bigr)
\\ & \quad 
=\exp\left(\int_0^t\Divh u(s,\xi_s(y))ds\right)
\Bigg|_{y=\xi_t^{-1}(z)}
\exp\left(\sum_{i=1}^N \Divh a_i(\xi_t^{-1}(z)) W_t^i\right) 
\\ & \quad \quad \quad \quad \quad \quad \times  
\rho_0 \bigl(\gamma_{p_l}^{-1} \circ \xi^{-1}_t(z)\bigr),
\end{split}
\end{equation}
for $(t,z)\in [0,T]\times \R^d$, where 
$\xi$ is a stochastic flow of 
diffeomorphisms, satisfying
\begin{equation*}
d\xi_t(z) = -u(t,\xi_t(z))\, dt 
-\sum_{i=1}^Na_i(\xi_t(z))\circ dW^i_t, 
\qquad \xi_0(z)=z.
\end{equation*}
Here the vector fields $u,a_i$ are 
seen as vectors in $\R^n$ 
through our coordinate system. 

Let us derive an $L^p$ bound. 
Fix $p\in [1,\infty)$ and let $(\chi_l)_l$ 
be a smooth partition of unity 
subordinated to our atlas $\mathcal{A}$. We have
\begin{equation*}
\begin{split}
\int_M \chi_l(x) \abs{\rho(t,x)}^p\, dV_h(x) 
& = \int_{\supp \chi_l} \chi_l(x) 
\abs{\rho(t,x)}^p\, dV_h(x)
\\ & =\int_{\gamma_{p_l}(\supp \chi_l)} 
\chi_l \bigl(\gamma_{p_l}^{-1} (z)\bigr) 
\abs{\rho\bigl(t,\gamma_{p_l}^{-1} (z)\bigr)}^p 
\abs{h_{\gamma_{p_l}}(z)}^{1/2} \, dz,
\end{split}
\end{equation*}
where $\abs{h_{\gamma_{p_l}}}^{1/2}$ 
denotes the determinant of the metric $h$ 
written in the coordinates induced by $\gamma_{p_l}$. 
Using \eqref{eq:expression-strong-sol} 
and the change of variable $z=\xi_t(w)$, 
we obtain
\begin{equation*}
\begin{split}
\int_M &\chi_l(x) \abs{\rho(t,x)}^p\, dV_h(x) 
\\ &
= \int_{\gamma_{p_l}(\supp \chi_l)} 
\chi_l\bigl(\gamma_{p_l}^{-1} (z)\bigr) 
\exp\left(p\int_0^t\Divh u \bigl(s,\xi_s(y)\bigr)\, ds
\right)\Bigg|_{y=\xi_t^{-1}(z)}\\
& \qquad\times  
\exp\left(p\sum_{i=1}^N 
\Divh a_i\bigl(\xi_t^{-1}(z)\bigr) W_t^i\right) 
\,\abs{\rho_0\bigl(\gamma_{p_l}^{-1} \circ \xi^{-1}_t(z)
\bigr)}^p\abs{h_{\gamma_{p_l}}(z)}^{1/2} \, dz
\\ & 
= \int_{\xi_t^{-1}\circ\gamma_{p_l}(\supp \chi_l)} 
\chi_l\bigl(\gamma_{p_l}^{-1}\circ \xi_t(w)\bigr) 
\exp\left(p\int_0^t\Divh u\bigl(s,\xi_s(w)\bigr)\, ds\right) 
\\ & \qquad \times  
\exp\left(p \sum_{i=1}^N \Divh a_i(w) W_t^i\right)  
\,\abs{\rho_0(\gamma_{p_l}^{-1}(w))}^p \abs{\partial\xi_t(w)}
\abs{h_{\gamma_{p_l}}(\xi_t(w)}^{1/2}\, dw
\\ &= \int_{\xi_t^{-1}\circ\gamma_{p_l}(\supp \chi_l)} 
\chi_l\bigl(\gamma_{p_l}^{-1}\circ \xi_t(w)\bigr) 
\exp\left(p\int_0^t\Divh u \bigl(s,\xi_s(w)\bigr)\, ds\right)
\\ & \qquad\times  
\exp\left(p\sum_{i=1}^N \Divh a_i(w) W_t^i\right)
\,\abs{\rho_0\bigl(\gamma_{p_l}^{-1}(w)\bigr)}^p
\abs{h_{\xi_t^{-1}\circ \gamma_{p_l}}(w)}^{1/2}
\, dw.
\end{split}
\end{equation*}
In passing, note that $\xi_t^{-1}\circ\gamma_{p_l}$ 
is a bona fide smooth chart. In the following, 
$C$ denotes a constant that depends only on 
$T$, $p$, $\norm{\Divh u}_{L^\infty_{t,x}}$, 
$\norm{\Divh a_i}_{L^\infty}$ and is allowed 
to vary from line to line. For convenience, set 
$A_i:=\norm{\Divh a_i}_{L^\infty(M)}$. 
We proceed as follows:
\begin{equation*}
\begin{split}
\int_M &\chi_l(x) \abs{\rho(t,x)}^p\, dV_h(x) 
\\ & \leq 
C \int_{\xi_t^{-1}\circ\gamma_{p_l}(\supp \chi_l)} 
\chi_l\bigl(\gamma_{p_l}^{-1}\circ \xi_t(w)\bigr) 
\\ & \qquad\qquad\qquad 
\times \exp\left(p \sum_{i=1}^N A_i \abs{W_t^i}\right)  
\,\abs{\rho_0\bigl(\gamma_{p_l}^{-1}(w)\bigr)}^p 
\abs{h_{\xi_t^{-1}\circ\gamma_{p_l}}(w)}^{1/2}\, dw
\\ & = C\exp\left(p\sum_{i=1}^N A_i 
\abs{W_t^i}\right) \int_{\xi_t^{-1}
\circ\gamma_{p_l}(\supp \chi_l)}
\chi_l\bigl(\gamma_{p_l}^{-1}\circ \xi_t(w)\bigr)
\\ & \qquad \qquad \qquad\times
\, \abs{\rho_0\bigl(\gamma_{p_l}^{-1}(w)\bigr)}^p
\abs{h_{\xi_t^{-1}\circ\gamma_{p_l}}(w)}^{1/2} \, dw
\\ & \leq C  \exp\left(p\sum_{i=1}^N A_i \abs{W_t^i}\right) 
\norm{\rho_0}^p_{L^\infty(M)}
\\ & \qquad \qquad \qquad \times 
\int_{\xi_t^{-1}\circ\gamma_{p_l}(\supp \chi_l)}
\chi_l\bigl(\gamma_{p_l}^{-1}\circ \xi_t(w)\bigr)
\abs{h_{\xi_t^{-1}\circ\gamma_{p_l}}(w)}^{1/2}\, dw
\\ & 
= C\exp\left(p \sum_{i=1}^N A_i \abs{W_t^i}\right)
\norm{\rho_0}^p_{L^\infty(M)}
\int_{\supp \chi_l} \chi_l(x)\, dV_h(x).
\end{split}
\end{equation*}
Taking expectation leads to 
\begin{equation*}
\begin{split}
\EE & \int_M \chi_l(x) \abs{\rho(t,x)}^p\, dV_h(x) 
\\ & \leq C  \norm{\rho_0}^p_{L^\infty(M)} \int_M \chi_l(x)
\, dV_h(x)\, \EE \exp\left(p\sum_{i=1}^N A_i \abs{W_t^i}\right)
\\ & = C  \norm{\rho_0}^p_{L^\infty(M)} 
\int_M \chi_l(x)\, dV_h(x) \,
\prod_{i=1}^N \EE \exp\left(p A_i \abs{W_t^i}\right) 
\\ & \leq C  \norm{\rho_0}^p_{L^\infty(M)} 
\int_M \chi_l(x)\, dV_h(x),
\end{split}
\end{equation*}
where we have used that the Brownian motions are 
independent and satisfy the standard 
estimate \cite[page 54]{Friedman}
\begin{equation*}
\EE \exp \bigl(\alpha \abs{W_t^i}\bigr) 
\leq \beta,\qquad 
t\in [0,T],\,\,\alpha >0,
\end{equation*}
where the constant $\beta$ 
depends on $\alpha$ and $T$. 
Therefore, summing over $l$, we obtain
\begin{equation*}
\EE\norm{\rho(t)}_{L^p(M)}^p \leq 
C \norm{\rho_0}^p_{L^\infty(M)} \int_M \, dV_h(x)
=C\norm{\rho_0}^p_{L^\infty(M)},
\end{equation*}
where the constant $C$ depends on the $L^\infty$ norms 
of $\Divh u$, $\Div a_1,\ldots,\Divh a_N$. 
Since we are assuming that $\rho_0,u\in C^\infty$, 
the right-hand side of the last expression is finite, 
and thus $\rho\in L^\infty_tL^p_{\omega,x}$. 
Moreover, using \cite[Theorem 1.1]{Kunita82}, 
we infer that the stochastic process $(\omega,t)
\mapsto \int_M\rho(t)\psi\,dV_h$ is a 
continuous $\mathcal{F}_t$-semimartingale 
for any $\psi\in C^\infty(M)$. 

Let us summarize all these results in 
\begin{lemma}[$L^p$ estimates, smooth data]
\label{lem:est-strong-sol}
Suppose $\rho_0, u\in C^\infty$. 
Let $\rho$ be the unique strong 
solution of \eqref{eq:target} given 
by Lemma \ref{lem:exist-strong-sol}. 
Then, for any $p\in [1,\infty)$,
$$
\rho\in L^\infty\left([0,T];L^p(\Omega\times M)\right),
\quad \sup_{t\in [0,T]}
\EE\norm{\rho(t)}_{L^p(M)}^p 
\leq C\norm{\rho_0}^p_{L^\infty(M)},
$$
where $C=C\left(p,T,\norm{\Divh u}_{L^\infty([0,T]\times M)},
\max_i\norm{\Divh a_i}_{L^\infty(M)}\right)$. 
Besides, for any $\psi\in C^\infty(M)$, the 
process $(\omega,t)\mapsto \int_M\rho(t)\psi\,dV_h$ 
is a continuous $\mathcal{F}_t$-semimartingale.
\end{lemma}

Let us bring \eqref{eq:target} into its It\^o form, 
still assuming that $\rho_0,u\in C^\infty$. 
We are not going to spell out all the details, 
referring instead to \cite{Kunita82} 
for the missing pieces. The solution $\rho$ 
we have constructed in Lemma \ref{lem:exist-strong-sol} 
is a smooth semimartingale, and it 
satisfies $\P$-a.s.~the following equation:
\begin{equation}\label{eq:convert-1}
\begin{split}
\rho(t,x) & = \rho_0(x) - \int_0^t 
\Divh \left (\rho(s,x)\, u \right)\,ds 
- \sum_{i=1}^N\int_0^t 
\Divh \left(\rho(s,x) \,a_i \right) \,dW^i(s)
\\ & \qquad 
-\frac{1}{2}\sum_{i=1}^N 
\left\langle\Divh \left(\rho(\cdot,x)a_i\right),
W_\cdot^i\right\rangle_t
\\ & 
= \rho_0(x) - \int_0^t \Divh \left(\rho(s,x)\, u \right)\,ds
-\sum_{i=1}^N\int_0^t \Divh \left(\rho(s,x) \,a_i\right)\,dW^i(s)
\\ &\qquad 
-\frac{1}{2}\sum_{i=1}^N \left\langle a_i(\rho(\cdot,x)),
W_\cdot^i\right\rangle_t
-\frac{1}{2}\sum_{i=1}^N \left\langle \rho(\cdot,x),
W_\cdot^i\right\rangle_t \Divh a_i,
\end{split}
\end{equation}
for all $t\in [0,T]$ and $x\in M$, by 
definition of the Stratonovich integral. 
By Theorem 1.1 and Lemma 1.3 in \cite{Kunita82}, 
we obtain 
\begin{equation*}
\begin{split}
& a_i(\rho(t,x)) = a_i(\rho_0(x)) 
-\int_0^t a_i\bigl(\Divh(\rho(s,x)\, u)\bigr)\,ds 
\\ &\quad -\sum_{j=1}^N\int_0^t a_i
\bigl(\Divh( \rho(s,x) \,a_j)\bigr) \,dW^j(s)
-\frac{1}{2}\sum_{j=1}^N \left\langle
a_i\bigl(\Divh(\rho(\cdot,x)a_j)\bigr),W_\cdot^j
\right\rangle_t,
\end{split}
\end{equation*}
and 
\begin{equation}\label{eq:convert-3}
\begin{split}
\left\langle a_i(\rho(\cdot,x)),W^i_\cdot \right\rangle_t
&=- \sum_{j=1}^N\left\langle \int_0^\cdot 
a_i\bigl(\Divh( \rho(s,x) \,a_j)\bigr) \,dW^j(s),
W^i_\cdot\right\rangle_t
\\ & = - \sum_{j=1}^N \int_0^t 
a_i\bigl(\Divh \left( \rho(s,x) \,a_j\right)\bigr) 
\, d\left\langle W^j,W^i\right\rangle_s
\\ &= - \int_0^t a_i\bigl(\Divh \left( \rho(s,x) \,a_i\right)
\bigr) \,ds,
\end{split}
\end{equation}
because the Brownian motions are independent, and 
the time-integral involving $u$ is absolutely 
continuous and thus not 
contributing to the quadratic variation.

Moreover, it is clear that
\begin{equation}\label{eq:convert-4}
\begin{split}
\left\langle \rho(\cdot,x),W^i_\cdot \right\rangle_t 
&= - \sum_{j=1}^N\left\langle \int_0^\cdot 
\Divh \left( \rho(s,x) \,a_j \right) \,dW^j(s),
W^i_\cdot\right\rangle_t
\\ &= - \sum_{j=1}^N \int_0^t 
\Divh \left( \rho(s,x) \,a_j \right) 
\, d\langle W^j,W^i\rangle_s
\\ &= - \int_0^t 
\Divh \left( \rho(s,x) \,a_i \right) \,ds.
\end{split}
\end{equation}
Re-starting from \eqref{eq:convert-1}, 
using \eqref{eq:convert-3} and \eqref{eq:convert-4}, 
we finally arrive at 
\begin{equation*}
\begin{split}
\rho(t,x) &= \rho_0(x) 
-\int_0^t \Divh \left(\rho(s,x)\, u \right)\,ds 
-\sum_{i=1}^N\int_0^t \Divh \left( \rho(s,x) \,a_i \right) 
\, dW^i(s) 
\\ & \qquad\qquad 
+\frac{1}{2}\sum_{i=1}^N \int_0^t 
a_i\bigl(\Divh \left(\rho(s,x)\,a_i\right) \bigr)\,ds
\\ & \qquad\qquad
+\frac{1}{2}\sum_{i=1}^N \int_0^t 
\Divh a_i\,\Divh \left(\rho(s,x)\,a_i \right)\,ds
\\ & = \rho_0(x) - \int_0^t \Divh \left(\rho(s,x)\, u\right)\,ds 
- \sum_{i=1}^N\int_0^t \Divh \left( \rho(s,x) \,a_i \right) 
\,dW^i(s) \\ & \qquad \qquad 
+\frac{1}{2}\sum_{i=1}^N 
\int_0^t \Lambda_i(\rho(s,x))\,ds,
\end{split}
\end{equation*}
where the second order differential equation 
$\Lambda_i$ is defined in \eqref{eq:def-Lambda-i}. 
This is the strong It\^o form of \eqref{eq:target}, 
derived under the assumption that $\rho_0,u\in C^\infty$. 
If we now integrate this against $\psi\in C^\infty(M)$ (say), 
since the It\^o integral admits a Fubini-type theorem, 
we arrive at the weak form given in 
Definition \ref{def:L2-weak-sol-Ito}.

In view of this, combining Lemmas \ref{lem:exist-strong-sol} 
and \ref{lem:est-strong-sol}, we eventually arrive at

\begin{proposition}[weak solution, smooth data]
\label{prop:strongsol-is-weaksol}
Let $\rho$ given by \eqref{eq:strong-sol-defining-eqn} 
be the unique strong solution of \eqref{eq:target} 
with initial datum $\rho_0\in C^\infty(M)$ and 
smooth vector field $u:[0,\infty)\times M\to TM$. 
Then $\rho$ is a weak $L^2$ solution of \eqref{eq:target} 
in the sense of Definition \ref{def:L2-weak-sol}.
\end{proposition}


\section{Time-dependent test functions}
\label{sec:time-dependent-test}

During an upcoming proof (of the $L^2$ estimate), we 
will need a version of the weak formulation 
\eqref{eq:renorm-weak-x-form} that makes  
use of time-dependent test functions. 
The next result supplies that formulation.

\begin{lemma}[space-time weak formulation]
\label{lem:time-dependent-test}
Let $\rho$ be a weak $L^2$ solution of \eqref{eq:target} 
with initial datum $\rho|_{t=0}=\rho_0$. Suppose 
$\rho$ is renormalizable in the sense of 
Definition \ref{def:renorm}. Fix $F\in C^2(\R)$ 
with $F,F',F''\in L^\infty(\R)$. 
For any $\psi\in C^\infty([0,T]\times M)$, 
the following equation holds $\P$-a.s., for any $t\in [0,T]$,
\begin{equation}\label{eq:time-dependent-test}
\begin{split}
&\int_M F(\rho(t))\psi(t) \, dV_h 
- \int_M F(\rho_0)\psi(0)\, dV_h 
\\ & \, 
=\int_0^t\int_M F(\rho(s))\partial_t\psi\, dV_h\,ds 
+\int_0^t\int_M F(\rho(s))\,u(\psi)\,dV_h\,ds
\\ &\quad 
+\sum_{i=1}^N\int_0^t\int_M 
F(\rho(s)) \,a_i(\psi) \, dV_h\, dW^i(s)
+ \frac12 \sum_{i=1}^N \int_0^t \int_M 
F(\rho(s))\, a_i\bigl(a_i(\psi)\bigr) \, dV_h \,ds
\\ & \quad 
-\int_0^t\int_M G_F(\rho(s))\Divh u \,\psi\, dV_h\, ds
-\sum_{i=1}^N\int_0^t\int_MG_F(\rho(s))\Divh a_i 
\,\psi\,dV_h \, dW^i(s) 
\\ & \quad 
-\frac12 \sum_{i=1}^N\int_0^t\int_M 
\Lambda_i(1)\,G_F(\rho(s))\,\psi\,dV_h\,ds
\\ & \quad 
+\frac12\sum_{i=1}^N\int_0^t\int_M 
F''(\rho(s))\bigl(\rho(s)\Divh a_i\bigr)^2\,\psi\,dV_h\,ds
\\ & \quad
-\sum_{i=1}^N\int_0^t\int_M G_F(\rho(s)) 
\bar{a}_i(\psi)\,dV_h \,ds.
\end{split}
\end{equation}
\end{lemma}

\begin{proof}
It is sufficient to consider test functions of the form 
$\psi(t,x)=\theta(t)\phi(x)$, where $\theta\in C^1_c((-1,T+1))$ 
and $\phi\in C^\infty(M)$, because the general 
result will then follow from a density argument for 
the tensor product. We start off from the following 
space-weak formulation, see \eqref{eq:renorm-weak-x-form}:
{\small
\begin{equation*}
\begin{split}
&\int_M F(\rho(t))\phi \, dV_h 
= \int_M F(\rho_0)\phi\, dV_h 
+\int_0^t\int_M F(\rho(s))\,u(\phi)\,dV_h\,ds
\\ &\quad 
+\sum_{i=1}^N\int_0^t\int_M 
F(\rho(s)) \,a_i(\phi) \, dV_h\, dW^i(s)
+ \frac12 \sum_{i=1}^N \int_0^t \int_M 
F(\rho(s))\, a_i\bigl(a_i(\phi)\bigr) \, dV_h \,ds
\\ & \quad 
-\int_0^t\int_M G_F(\rho(s))\Divh u \,\phi\, dV_h\, ds
-\sum_{i=1}^N\int_0^t\int_MG_F(\rho(s))\Divh a_i 
\,\phi\,dV_h \, dW^i(s) 
\\ & \quad 
-\frac12 \sum_{i=1}^N\int_0^t\int_M 
\Lambda_i(1)\,G_F(\rho(s))\,\phi\,dV_h\,ds
\\ & \quad 
+\frac12\sum_{i=1}^N\int_0^t\int_M 
F''(\rho(s))\bigl(\rho(s)\Divh a_i\bigr)^2\,\phi\,dV_h\,ds
\\ & \quad
-\sum_{i=1}^N\int_0^t\int_M G_F(\rho(s)) 
\bar{a}_i(\phi)\,dV_h \,ds, 
\qquad \text{$\P$-a.s., for any $t\in [0,T]$.}
\end{split}
\end{equation*}
}
We multiply this equation by $\Dot{\theta}(t)$ 
and integrate the result over $t\in \left[0, \bar{t} \right]$. 
All the time-integrals are absolutely continuous 
by definition, and thus we can integrate 
them by parts. For example, 
\begin{equation*}
\begin{split}
\int_0^{\bar{t}} & \Dot{\theta}(t)\int_0^t\int_M 
F(\rho(s))\,u(\phi)\,dV_h\,ds\,dt 
\\ &= \theta({\bar{t}})\int_0^{\bar{t}}
\int_M F(\rho(s))\,u(\phi)\,dV_h\,ds 
- \int_0^{\bar{t}} \theta(t)
\int_M F(\rho(t))\,u(\phi)\,dV_h\,dt,
\end{split}
\end{equation*}
and so forth. We can also integrate by parts the 
stochastic integrals. For example,
\begin{equation*}
\begin{split}
\int_0^{\bar{t}} & \Dot{\theta}(t) \int_0^t 
\int_M	F(\rho(s)) \,a_i(\phi) \, dV_h\, dW^i(s)\,dt
\\ & =\theta({\Bar{t}}) \int_0^{\Bar{t}} 
\int_M	F(\rho(s)) \,a_i(\phi) \, dV_h\, dW^i(s)
- \int_0^{\Bar{t}}\theta(t) 
\int_M	F(\rho(s)) \,a_i(\phi) \, dV_h\, dW^i(t),
\end{split}
\end{equation*}
and so forth. Finally,
\begin{align*}
& \int_0^{\bar t}\Dot{\theta}(t)
\left(\int_M F(\rho(t))\phi \, dV_h 
- \int_M F(\rho_0)\phi\, dV_h\right)\, dt
\\ & \quad = \int_0^{\bar t}
\int_M F(\rho(t)) \Dot{\theta}(t)\phi \, dV_h\, dt
+\int_M F(\rho_0) \theta(0)\phi\, dV_h
-\int_M F(\rho_0) \theta(\bar t)\phi\, dV_h,
\end{align*}
where the last term is aggregated together 
with the other ``$\theta({\Bar{t}}) \int_0^{\Bar{t}}
\left(\cdots\right)$" terms that appear, eventually leading to 
$\int_M F(\rho(\bar t))\theta(\bar t)\phi \, dV_h$. 
Therefore, after many straightforward rearrangements 
of terms, we arrive at (now replacing $\bar t$ by $t$)
{\small
\begin{equation*}
\begin{split}
&\int_M F(\rho(t))\theta(t)\phi \, dV_h 
- \int_M F(\rho_0)\theta(0)\phi\, dV_h  
=\int_0^t\int_M F(\rho(s))\Dot{\theta}(s)\phi\, dV_h\,ds 
\\ &\quad 
+\int_0^t\int_M F(\rho(s))\,u(\theta(s)\phi)\,dV_h\,ds
+\sum_{i=1}^N\int_0^t\int_M 
F(\rho(s)) \,a_i(\theta(s)\phi) \, dV_h\, dW^i(s)
\\ &\quad 
+ \frac12 \sum_{i=1}^N \int_0^t \int_M 
F(\rho(s))\, a_i\bigl(a_i(\theta(s)\phi)\bigr) \, dV_h \,ds
-\int_0^t\int_M G_F(\rho(s))\Divh u \,\theta(s)\phi\, dV_h\, ds
\\ & \quad 
-\sum_{i=1}^N\int_0^t\int_MG_F(\rho(s))\Divh a_i 
\,\theta(s)\phi\,dV_h \, dW^i(s) 
\\ & \quad 
-\frac12 \sum_{i=1}^N\int_0^t\int_M 
\Lambda_i(1)\,G_F(\rho(s))\,\theta(s)\phi\,dV_h\,ds
\\ & \quad 
+\frac12\sum_{i=1}^N\int_0^t\int_M 
F''(\rho(s))\bigl(\rho(s)\Divh a_i\bigr)^2\,\theta(s)\phi\,dV_h\,ds
\\ & \quad
-\sum_{i=1}^N\int_0^t\int_M G_F(\rho(s)) 
\bar{a}_i(\theta(s)\phi)\,dV_h \,ds.
\end{split}
\end{equation*}
}By density of tensor products \cite{deRham}, this equation 
continues to hold for any test function 
$\psi\in C^\infty_c((-1,T+1)\times M)$ 
and thus for any $\psi\in C^\infty([0,T]\times M)$.
\end{proof}


\section{Irregular test functions}\label{sec:non-smooth-test}
We need to insert into the 
weak formulation \eqref{eq:time-dependent-test} 
test functions $\psi(t,x)$ that are non-smooth. 
Clearly, in view of our assumptions, the 
stochastic integrals in \eqref{eq:time-dependent-test}
are zero-mean martingales. Hence, after taking the expectation 
in \eqref{eq:time-dependent-test}, we obtain
\begin{equation}\label{eq:renorm-irregular-test}
\begin{split}
& \EE\int_M F(\rho(t))\psi(t) \, dV_h 
- \EE\int_M F(\rho_0)\psi(0) \, dV_h  
\\ &\quad =
\EE\int_0^t\int_M F(\rho(s))\partial_t\psi \, dV_h\,ds 
+\EE\int_0^t\int_M F(\rho(s))\,u(\psi) \,dV_h\,ds
\\ & \quad \qquad
+\frac12 \sum_{i=1}^N \EE\int_0^t \int_M 
F(\rho(s))\, a_i\bigl(a_i(\psi)\bigr) \, dV_h \,ds
\\ & \quad \qquad
-\EE\int_0^t\int_M G_F(\rho(s))\Divh u \,\psi\, dV_h\, ds
\\ & \quad \qquad
-\frac12 \sum_{i=1}^N\EE\int_0^t\int_M 
\Lambda_i(1)\,G_F(\rho(s))\,\psi\,dV_h\,ds
\\ & \quad \qquad
+\frac12\sum_{i=1}^N\EE\int_0^t\int_M 
F''(\rho(s))\bigl(\rho(s)\Divh a_i\bigr)^2
\,\psi\,dV_h\,ds
\\ & \quad \qquad
-\sum_{i=1}^N\EE\int_0^t\int_M G_F(\rho(s)) 
\bar{a}_i(\psi)\,dV_h \,ds,
\end{split}
\end{equation}
which holds for any test function 
$\psi\in C^\infty([0,T]\times M)$.

The main result of this section is

\begin{lemma}[non-smooth test functions]
\label{lemma:irregular-test-Gronwall}
Let $\rho$ be a weak $L^2$ solution 
of \eqref{eq:target} with initial datum $\rho|_{t=0}=\rho_0$
and assume that $\rho$ is renormalizable. 
Fix $F\in C^2(\R)$ with $F,F',F''\in L^\infty(\R)$. 
Fix a time $t_0\in (0,T]$ and consider 
\eqref{eq:renorm-irregular-test} evaluated at $t=t_0$. 
Then \eqref{eq:renorm-irregular-test} continues to hold 
for any $\psi\in W^{1,2,p}([0,t_0]\times M)$ with $p>d+2$.
\end{lemma}

\begin{proof}
By Proposition \ref{prop:anisotropic-embed}, 
$W^{1,2,p}([0,t_0]\times M)$ compactly embeds into 
$C^0([0,t_0]\times M)$ (since $p>d+2$). 
Moreover, the first order $x$-derivatives of 
a $W^{1,2,p}$ function belong to $C^0([0,t_0]\times M)$. 
Therefore, given a function $\psi \in W^{1,2,p}([0,t_0]\times M)$, 
the very definition of $W^{1,2,p}$ implies the 
existence of a sequence $\seq{\psi_j}_{j\ge 1}
\subset C^\infty([0,t_0]\times M)$ such 
that $\psi_j\to \psi$ in $W^{1,2,p}([0,t_0]\times M)$.
Besides, we have
$$
\psi_j\to \psi, 
\quad 
\nabla\psi_j\to \nabla\psi  
\quad 
\text{uniformly on $[0,t_0]\times M$}.
$$
We extend the functions $\psi_j$ to $C^\infty([0,T]\times M)$ 
by means of Proposition \ref{prop:Seeley}. These 
extensions are also denoted by $\psi_j$. 
Consequently, we can insert $\psi_j$ 
into \eqref{eq:renorm-irregular-test}.  

Equipped with the above convergences and the assumptions 
$\rho\in L^\infty_tL^2_{\omega,x}$ and 
$u\in L^1_t\overrightarrow{W_x^{1,2}}$, it is straightforward 
(repeated applications of H\"older's inequality) to
verify that \eqref{eq:renorm-irregular-test} 
holds for test functions $\psi$ 
that belong to $W^{1,2,p}([0,t_0]\times M)$.
\end{proof}


\section{On the ellipticity of $\sum_i a_i(a_i)$, 
proof of Lemma \ref{lem:ellipticity}}\label{sec: ellipticity}

In this section we will prove Lemma \ref{lem:ellipticity}. 
Before doing that, however, let us explain why the 
second order differential operator 
$\sum_i a_i\bigl(a_i(\cdot)\bigr)$, in general, fails 
to be non-degenerate (elliptic). To this end, we introduce 
the following (smooth) sections of the endomorphisms over $TM$:
\begin{equation}\label{eq:def-Ai}
\mathcal{A}_i(x)X := 
\bigl(X,a_i(x)\bigr)_h\,a_i(x),
\qquad x\in M,\,\, X\in T_xM,
\quad i=1,\ldots,N. 
\end{equation}
It is clear that these sections are 
symmetric with respect to $h$, namely 
$$
\bigl(\mathcal{A}_i(x)X,Y\bigr)_h
=\bigl(X,\mathcal{A}_i(x)Y\bigr)_h,
\qquad x\in M,\,\, X,Y\in T_xM.
$$
Set 
\begin{equation}\label{eq:def-A}
\mathcal{A}:=\mathcal{A}_1+\cdots +\mathcal{A}_N,
\end{equation}
which is still a smooth section of 
the symmetric endomorphisms over $TM$. 
Given the sections $\mathcal{A}_1,\ldots, 
\mathcal{A}_N$ and $\mathcal{A}$, we define the following 
second order linear differential operators in divergence form:
\begin{align*}
& C^2(M)\ni \psi \mapsto 
\Divh \left(\mathcal{A}_i\nabla_h\psi\right),
\qquad i=1,\ldots ,N,
\\ &
C^2(M)\ni \psi \mapsto 
\Divh \left(\mathcal{A}\nabla_h\psi\right) 
= \sum_{i=1}^N \Divh(\mathcal{A}_i\nabla_h\psi).
\end{align*}

Observe that the following 
identity holds trivially:
$$
a_i\bigl(a_i(\psi)\bigr) = 
\Divh \left(\mathcal{A}_i\nabla_h\psi\right) 
- \bar{a}_i(\psi),
$$
thus 
\begin{equation}\label{eq:geom-id-aa_i-A}
\sum_{i=1}^N a_i\bigl(a_i(\psi)\bigr) 
=\Divh \left(\mathcal{A}\nabla_h\psi \right) 
-\sum_{i=1}^N\bar{a}_i(\psi),
\end{equation}
where $\bar{a}_i$ is short-hand for 
the first order differential operator $(\Divh a_i)\, a_i$. 
Thus $\sum_{i=1}^N a_i\bigl(a_i(\cdot)\bigr)$ 
is non-degenerate (elliptic) if and only if 
$\Divh \left(\mathcal{A}\nabla_h\cdot\right)$ is so.

In view of \eqref{eq:geom-id-aa_i-A}, 
let us see why the induced differential 
operator $\sum_{i=1}^N a_i\bigl(a_i(\cdot)\bigr)$ 
may degenerate. From the very definition 
of $\mathcal{A}$, we have 
$$
\bigl(\mathcal{A}(x)X,X\bigr)_h 
=\sum_{i=1}^N \bigl(\mathcal{A}_i(x)X,X\bigr)_h 
= \sum_{i=1}^N \bigl(X,a_i(x)\bigr)_h^2, 
\qquad x\in M, \, \, X\in T_xM, 
$$
and the last expression may be zero unless 
we can find vector fields $a_{i_1}(x),\ldots,a_{i_d}(x)$ 
that constitute a basis for $T_xM$. Note that this can 
also happen in the ``ideal'' case $N=d$, 
that is, one can always find suitable $x\in M$ 
and $X\in T_xM$ such that $\bigl(\mathcal{A}(x)X,X\bigr)_h =0$. 
The explanation for this fact is geometric in nature. 
In general, given an arbitrary $d$-dimensional 
smooth manifold $M$, it is not possible to construct 
a global frame, i.e., smooth vector fields 
$E_1,\ldots ,E_d$ forming a basis for $T_xM$ for all $x\in M$. 
If this happens, the manifold is called 
\textit{parallelizable}. Examples of 
parallelizable manifolds are Lie groups (like $\R^d$, 
$\mathbb{T}^d$) and $\mathbb{S}^d$ 
with $d\in \seq{1,3,7}$.

Nevertheless, by compactness of $M$, one can 
always find vector fields $a_1,\ldots ,a_N$ 
with $N\geq d$, depending on the geometry of $M$, 
such that the resulting operator 
$\Divh \left(\mathcal{A}\nabla_h\cdot\right)$ 
becomes the Laplace-Beltrami operator (and hence elliptic). 
In other words, to implement our strategy of 
using noise to avoid density concentrations, 
we will add to the original SPDE 
\eqref{eq:target} as many independent Wiener 
processes and first order differential operators 
$\bar a_1,\ldots ,\bar a_N$ as deemed necessary 
by the geometry of the manifold itself. 
Note that in the Euclidean case 
\cite{Attanasio:2011fj,Beck:2019,Flandoli-Gubinelli-Priola} 
one can always resort to the canonical differential operators 
$a_i=\partial_i$ and thus $\sum_{i=1}^N a_i\bigl(a_i(\cdot)\bigr) 
=\Divh \left(\mathcal{A}\nabla_h\cdot \right)=\Delta \cdot$ 
(with $N=d$). This simple approach does 
not work for us because of the Riemannian structure 
of the underlying domain $M$.

Having said all of that, let us now return to the proof 
of Lemma \ref{lem:ellipticity}, which will be a 
trivial consequence of the following crucial result:

\begin{lemma}\label{lem:exist-ai}
There exist $N=N(M)$ smooth vector fields 
$a_1,\ldots ,a_N$ on $M$ such that the corresponding 
section $\mathcal{A}$, see \eqref{eq:def-Ai} 
and \eqref{eq:def-A}, satisfies
$$
\bigl(\mathcal{A}(x)X,Y\bigr)_h 
= 2\left(X,Y\right)_h, 
\qquad \forall x\in M, \, \, 
\forall X,Y\in T_xM.
$$
Consequently, $\mathcal{A}(x) = 2\,I_{T_xM}$ 
for all $x\in M$. 
\end{lemma}

\begin{proof}
Let $p\in M$. Then, by means of the Gram-Schmidt algorithm, 
we can easily construct a local orthonormal 
frame near $p$, that is, a local frame 
$E_{p,1},\ldots , E_{p,d}$ defined 
in an open neighborhood $\mathcal{U}_p$ of $p$ 
that forms an orthonormal basis for the tangent 
space at each point of the neighborhood 
(see \cite[p.~24]{LeeRiemann} for details). Since 
$\seq{\mathcal{U}_p}_{p\in M}$ forms an open 
covering of $M$, the compactness of $M$ ensures 
the existence of $p_1,\ldots ,p_L\in M$ such 
that $\bigcup_{j=1}^L\mathcal{U}_{p_j}=M$ 
and a collection of locally smooth vector fields 
$\seq{E_{p_j,i}}_{\overset{i=1,\ldots ,d}{j=1,\ldots ,L}}$ 
with the aforementioned property. Let us now consider 
a smooth partition of unity subordinate 
to $\seq{\mathcal{U}_{p_j}}_{j=1}^L$, which we 
may write as $\seq{\alpha_j^2}_{j=1}^L$, where 
$\alpha_j\in C^\infty(M)$ and $\sum_{j=1}^L\alpha^2_j=1$. 
Set $\tilde{E}_{p_j,i}:=\alpha_j E_{p_j,i}$, 
for $i=1,\ldots ,d$ and $j=1,\ldots ,L$. 
Extending these vector fields by 
zero outside their supports, we 
obtain global smooth vector fields on $M$. 

Observe that if $\alpha_j(x)\neq 0$, then 
$x\in \left(\supp \alpha_j\right)^\circ
=\left(\supp \alpha^2_j\right)^\circ
\subset \mathcal{U}_{p_j}$. As a result, 
$E_{p_j,1}(x),\ldots , E_{p_j,d}(x)$ constitute 
an orthonormal basis for $T_xM$.  
For convenience, we rename the vector fields 
$\seq{\tilde{E}_{p_j,i}}_{\overset{i=1,\cdots ,d}{j=1,\cdots ,L}}$ 
as $\beta_1,\ldots , \beta_N$, where $N:=d\cdot L$.

As before, we define sections $\mathcal{B}_i$ of the 
endomorphisms over $TM$ by setting
$$
\mathcal{B}_i(x)X := 
\bigl(X,\beta_i(x)\bigr)_h\,\beta_i(x),
\qquad x\in M,\,\,  X\in T_xM,
$$
for $i=1,\ldots,N$, and $\mathcal{B}:=\mathcal{B}_1
+\cdots+\mathcal{B}_N$.

For an arbitrary $x\in M$ and $X\in T_xM$, we compute
\begin{align*}
\bigl(\mathcal{B}(x)X,X\bigr)_h 
& = \sum_{k=1}^N  \bigl(X,\beta_k(x)\bigr)_h^2
=\sum_{j=1}^L\sum_{i=1}^d  
\left(X,\tilde{E}_{p_j,i}(x)\right)_h^2
\\ & =\sum_{j: \alpha_j(x)\neq 0}
\, \sum_{i=1}^d  \left(X,E_{p_j,i}(x)\right)_h^2 \alpha^2_j(x)
\\ & =\sum_{j: \alpha_j(x)\neq 0}\alpha^2_j(x)
\sum_{i=1}^d  \left(X,E_{p_j,i}(x)\right)_h^2
\\ & =\sum_{j: \alpha_j(x)\neq 0}
\alpha^2_j(x) \abs{X}_h^2 = \abs{X}_h^2. 
\end{align*}
By the polarization identity for inner products 
and the symmetry of $\mathcal{B}$, this last 
equality implies that
$$
\bigl(\mathcal{B}(x)X,Y\bigr)_h 
= \left(X,Y\right)_h,
\qquad \forall x\in M, \,\, 
\forall X,Y\in T_xM,
$$
and thus $\mathcal{B}(x)= I_{T_xM}$. 

Setting $a_i:=\sqrt{2}\beta_i$, $i=1,\ldots,N$, 
concludes the proof of the lemma.
\end{proof}

\begin{proof}[Proof of Lemma \ref{lem:ellipticity}]
Fix $\psi\in C^2(M)$. In view of Lemma \ref{lem:exist-ai}, 
the identity \eqref{eq:geom-id-aa_i-A} becomes
$$
\sum_{i=1}^N a_i\bigl(a_i(\psi)\bigr) 
= 2\Divh \left(\nabla_h\psi\right) 
-\sum_{i=1}^N\bar{a}_i(\psi) 
=2\Delta_h\psi-\sum_{i=1}^N\bar{a}_i(\psi),
$$
where $\bar{a}_i=(\Divh a_i)\, a_i$.
\end{proof}

From now on, we will be using the vector 
fields $a_1,\ldots ,a_N$ constructed in 
Lemma \ref{lem:exist-ai}, in which case 
the It\^{o} SPDE \eqref{eq:target-Ito} becomes
\begin{equation}\label{eq:target-Ito-special-noise}
d\rho + \Divh \left(\rho\left[ u 
-\frac{1}{2} \sum_{i=1}^N 
\bar{a}_i\right]\right)\, dt 
+ \sum_{i=1}^N \Divh( \rho \,a_i) \, dW^i(t) 
- \Delta_h \rho \, dt= 0.
\end{equation}
The space-weak formulation of this SPDE is
\begin{equation*}
\begin{split}
& \int_M \rho(t)\psi\, dV_h 
= \int_M \rho_0\psi\, dV_h 
+\int_0^t\int_M\rho(s)
\left[\,u(\psi)-\frac12\sum_{i=1}^N \bar{a}_i(\psi) 
\right]\,dV_h \,ds \\ & \qquad
+ \sum_{i=1}^N\int_0^t\int_M \rho(s) 
\,a_i(\psi) \, dV_h\, dW^i(s)
+ \int_0^t \int_M \rho(s) \,\Delta_h\psi \, dV_h \,ds,
\end{split}
\end{equation*}
see Definition \ref{def:L2-weak-sol-Ito} 
and equation \eqref{eq:L2weak-sol-Ito}.


\section{Test function for duality method}

In this section we first construct 
a solution to the following parabolic 
Cauchy problem on the manifold $M$: given $0<t_0\leq T$, solve 
\begin{equation}\label{eq:aux-Cauchy-prob}
\begin{cases}
\partial_t v - \Delta_h v + b(t,x) v
= f(x,t) \;\; \mbox{on $[0,t_0]\times M$},\\
v(0,x) = 0 \;\; \mbox{on $M$},
\end{cases}
\end{equation}
where $b$ and $f$ are given irregular functions 
in $L^p([0,t_0]\times M)$ (with $p\ge 1$ 
to be fixed later). We follow the strategy 
outlined in \cite[p.~131]{Aubin} (for smooth $b, f$), 
making use of Fredholm theory and anisotropic 
Sobolev spaces. Toward the end of this section, we 
utilize the solution of \eqref{eq:aux-Cauchy-prob} 
to construct a test function that will form the core 
of a duality argument given in an upcoming section.

Consider the space $W^{1,2,p}_0([0,t_0]\times M)$, 
which is the subspace of functions in the 
anisotropic Sobolev space $W^{1,2,p}([0,t_0]\times M)$ 
vanishing at $t=0$. Let $L$ designate the 
heat operator on $M$, namely $L=\partial_t-\Delta_h$. 
According to \cite[Thm.~4.45]{Aubin}, $L$ is an 
isomorphism of $W^{1,2,p}_0([0,t_0]\times M)$ 
onto $L^p([0,t_0]\times M)$ for $1\leq p<\infty$. 
Consider the multiplication operator
$$
K_b: W^{1,2,p}_0([0,t_0]\times M)
\to L^p([0,t_0]\times M), 
\qquad v \stackrel{K_b}{\mapsto} bv.
$$
To guarantee that this operator is well-defined, 
we must assume $p>d+2$. In this way, in 
view of Proposition \ref{prop:anisotropic-embed}, 
$W^{1,2,p}_0([0,t_0]\times M)$ compactly 
embeds into $C^0([0,t_0]\times M)$ and the 
first order space-derivatives of 
$v\in W^{1,2,p}_0([0,t_0]\times M)$ are continuous 
on $[0,t_0]\times M$. It then follows that
$$
\int_0^{t_0}\int_M \abs{bv}^p\, dV_h\,dt 
\leq \norm{v}_{C^0}^p \int_0^{t_0}\int_M
\abs{b}^p\, dV_h\,dt, 
$$
guaranteeing that $K_b$ is well-defined.

\medskip
\texttt{Claim.} $K_b$ is compact.

\medskip

\noindent First of all, $K_b$ is continuous:
$$
\norm{K_b v}_{L^p}
\leq \norm{v}_{C^0}
\norm{b}_{L^p} 
\leq C \norm{v}_{W^{1,2,p}_0}
\norm{b}_{L^p},
$$
where $C>0$ is a constant coming from the 
anisotropic Sobolev embedding, consult 
Proposition \ref{prop:anisotropic-embed}. 
Clearly, $W^{1,2,p}_0([0,t_0]\times M)$ is reflexive, 
being a closed subspace of $W^{1,2,p}([0,t_0]\times M)$. 
Hence, to arrive at the claim, it is 
enough to prove that $K_b$ is completely continuous. 
Recall that a bounded linear operator $T:X\to Y$ 
between Banach spaces is called completely 
continuous if weakly convergent sequences 
in $X$ are mapped to strongly converging sequences in $Y$. 
Let $\seq{v_n}_{n\ge 1}$ be a sequence in 
$W^{1,2,p}_0([0,t_0]\times M)$ such that 
$v_n\rightharpoonup v\in W^{1,2,p}_0$. 
By the compact embedding $W^{1,2,p}_0
\subset\subset C^0$, $v_n\to v$ in $C^0$. Hence, 
$$
\norm{K_bv_n-K_bv}_{L^p}
\leq \norm{v_n-v}_{C^0}
\norm{b}_{L^p} \to 0,
$$
and so $K_b$ is completely continuous. 
This concludes the proof of the claim.

\medskip

Next, being an isomorphism, $L$ is a Fredholm 
operator from $W^{1,2,p}_0$ to $L^p$. 
This implies that $L+K_b$ is a Fredholm operator, 
with index $\ind \left(L+K_b\right)
=\ind\left(L\right)$, where trivially 
$\ind\left(L\right)=0$ ($L$ is invertible). 
Thus our goal is to verify the following claim.

\medskip

\texttt{Claim.} Either $\ker\left(L+K_b\right)$ is 
trivial or $\codim \left(R(L+K_b)\right)=0$. 

\medskip

\noindent If this claim holds, then we will 
be able to conclude that \eqref{eq:aux-Cauchy-prob} 
is solvable for any $f\in L^p([0,t_0]\times M)$.  
The proof of the claim is divided into 
three main steps.

\medskip

\texttt{Step 1:} $b\in C^\infty([0,t_0]\times M)$.

\medskip

\noindent Our aim is to show that $\ker\left(L+K_b\right)$ 
is trivial. Let $v\in W^{1,2,p}_0([0,t_0]\times M)$ 
solve $(L+K_b)v=0$. Since $p>d+2$ and $b$ is smooth, it 
follows from parabolic regularity theory that $v$ 
is (at least) in $C^{1,2}([0,t_0]\times M)$. 
Indeed, by the anisotropic Sobolev embedding 
(Proposition \ref{prop:anisotropic-embed}), 
$v\in C^{0,\gamma}([0,t_0]\times M)$ 
with $\gamma=1-\frac{1+d}{p}$. Therefore,
$$
Lv = -bv \in C^{0,\gamma}([0,t_0]\times M),
$$
and $v(0,\cdot)=0$ on $M$. 
Parabolic regularity theory (see 
e.g.~\cite[p.~130]{Aubin}) implies that 
$\partial_t v$ and the second derivatives 
of $v$ with respect to $x$ are H\"older continuous. 

By the chain rule, the function 
$\psi:=\frac{v^2}{2}$ satisfies
$$
L\psi = - \abs{\nabla_hv}_h^2
- 2b\psi \leq -2b\psi.
$$
Since $b$ is bounded and $\psi(0,x)=0$, the maximum 
principle (see \cite[Prop.~4.3]{Chow-Knopf}) 
implies that $\psi\leq 0$ everywhere. 
On the other hand, $\psi\geq 0$ by definition. 
It follows that $\psi\equiv 0$, and so $v\equiv 0$. 

Hence, given any $b\in C^\infty([0,t_0]\times M)$, the 
Cauchy problem \eqref{eq:aux-Cauchy-prob} admits 
a unique solution for any $f\in L^p([0,t_0]\times M)$.

\medskip

\texttt{Step 2:} A priori estimates (smooth data).

\medskip

\noindent Let us consider the more general problem 
\begin{equation}\label{eq:aux-Cauchy-prob2}
\begin{cases}
\partial_t v - \Delta_h v + b(t,x) v
= g(x,t) \;\; \mbox{on $[0,t_0]\times M$},\\
v(0,x) = c(x) \;\; \mbox{on $M$}.
\end{cases}
\end{equation}
where $b,g\in C^\infty([0,t_0]\times M)$ 
and $c\in C^\infty(M)$. This problem admits a unique 
solution $v\in C^{1,2}([0,t_0]\times M)$, 
given by $v=\tilde{v}+c$, where $\tilde{v}$ 
solves \eqref{eq:aux-Cauchy-prob} with right-hand side 
$f=g-cb + \Delta_hc\in L^p([0,t_0]\times M)$. 

From known a priori estimates for the heat equation 
on manifolds (see \cite[Thm.~4.45]{Aubin}), there is 
a constant $C_0=C_0(p,M)$ such that 
$(\tilde{v}=v-c)$,
\begin{align*}
\norm{v}_{W^{1,2,p}} & = \norm{\tilde{v}+c}_{W^{1,2,p}} 
\leq \norm{\tilde{v}}_{W^{1,2,p}}  
+ T\norm{c}_{W^{2,p}(M)}
\\ & 
\leq C_0\norm{g-bc+\Delta_hc}_{L^p} 
+ T\norm{c}_{W^{2,p}(M)}
\\ & 
\leq C_0\left[\norm{g}_{L^p} 
+ \norm{b}_{L^p}\norm{c}_{C^0(M)} 
+ T\norm{\Delta_hc}_{L^p(M)}\right]
+ T\norm{c}_{W^{2,p}(M)},
\end{align*}
where $W^{2,p}(M)$ denotes the standard 
Sobolev space on $(M,h)$, which embeds 
into $C^{0}(M)$ (recall $p>2+d$). Therefore, for a 
constant $C=C(p,M,T)$, we infer
\begin{equation}\label{eq:Sobol-est-gen-Cauchy-prob}
\norm{v}_{W^{1,2,p}} 
\leq C\left[\norm{g}_{L^p} 
+ \norm{c}_{W^{2,p}(M)} 
\bigl(\norm{b}_{L^p}+ 1\bigr)\right].
\end{equation}
Summarizing, the general Cauchy 
problem \eqref{eq:aux-Cauchy-prob2} with 
$b,g\in C^\infty([0,t_0]\times M)$ and 
$c\in C^\infty(M)$ admits a unique 
solution $v\in C^{1,2}([0,t_0]\times M)$ 
satisfying \eqref{eq:Sobol-est-gen-Cauchy-prob}.

\medskip

\texttt{Step 3:} Well-posedness of \eqref{eq:aux-Cauchy-prob2}, 
non-smooth $b,g$.

\medskip

\noindent The aim is to prove the well-posedness 
of \eqref{eq:aux-Cauchy-prob2}---and thus 
\eqref{eq:aux-Cauchy-prob}---for irregular $b$ and $g$ in 
$L^p([0,t_0]\times M)$. Since $C^\infty([0,t_0]\times M)$ 
is dense in $L^p([0,t_0]\times M)$ \cite[Thm.~2.9]{Aubin}, 
there exist sequences $\seq{b_n}_{n\ge1}$ 
and $\seq{g_n}_{n\ge 1}$ of smooth functions such that
$$
b_n\stackrel{L^p}{\to} b, 
\qquad 
g_n\stackrel{L^p}{\to} g.
$$
From the previous step, there 
exists a unique solution 
$v_n\in W^{1,2,p}([0,t_0]\times M)$ of 
$$
\begin{cases}
\partial_t v_n 
-\Delta_hv_n + b_n(t,x) v_n
= g_n(x,t) \;\; \mbox{on $[0,t_0]\times M$},\\
v_n(0,x) = c(x) \;\; \mbox{on $M$}.
\end{cases}
$$
In view of \eqref{eq:Sobol-est-gen-Cauchy-prob}, $\seq{v_n}_{n\ge1}$ 
is bounded in $W^{1,2,p}([0,t_0]\times M)$. 
Therefore, up to a subsequence, we may assume that 
$$
\begin{cases}
v_n\rightharpoonup v\in W^{1,2,p}([0,t_0]\times M),\\
v_n\to v \in C^0([0,t_0]\times M).
\end{cases}
$$
Given these convergences, it is easy to 
conclude that $v$ solves the Cauchy problem 
\eqref{eq:aux-Cauchy-prob2} 
with $b,g\in L^p([0,t_0]\times M)$ 
and $c\in C^\infty(M)$. 

We summarize our findings so far in 

\begin{proposition}[well-posedness of parabolic Cauchy problem, 
non-smooth data]
\label{prop:well-posed-Cauchy-prob}
Suppose $b$ and $g$ belong to $L^p([0,t_0]\times M)$. 
Then there exists a unique solution 
$v\in W^{1,2,p}([0,t_0]\times M)$ to the 
Cauchy problem \eqref{eq:aux-Cauchy-prob2} with 
initial data $c\in C^\infty(M)$. Furthermore, 
the a priori estimate \eqref{eq:Sobol-est-gen-Cauchy-prob} holds.
\end{proposition}

\begin{proof}
The only assertion that remains to be verified 
is the one about uniqueness, but uniqueness of the 
solution is an immediate consequence 
of \eqref{eq:Sobol-est-gen-Cauchy-prob}. 
\end{proof}

\begin{remark}
The ``non-smooth" quantifier in Proposition 
\ref{prop:well-posed-Cauchy-prob} refers to 
the functions $b$ and $g$ in \eqref{eq:aux-Cauchy-prob2}. 
In upcoming applications it is essential that $b,g$ are 
allowed to be irregular (but a smooth 
initial function $c$ is fine, like $c\equiv 1$).

\end{remark}

\medskip

Let us now consider the special 
Cauchy problem
\begin{equation}\label{eq:Cauchy-prob3}
\begin{cases}
\partial_t v-\Delta_h v + b(t,x) v
= -b(x,t) \;\; \mbox{on $[0,t_0]\times M$},\\
v(0,x) = 0 \;\; \mbox{on $M$},
\end{cases}
\end{equation}
with $b\in C^\infty([0,t_0]\times M)$ and 
$b\leq 0$. This problem corresponds to 
\eqref{eq:aux-Cauchy-prob2} with a nonnegative 
smooth source $g$ (namely, $g=-b\ge 0$). 

By the previous discussion, there 
exists a unique solution $v\in C^{1,2}([0,t_0]\times M)$ 
to \eqref{eq:Cauchy-prob3}. 
Clearly, we have
$$
\partial_t v -\Delta_h v \geq - b(t,x) v,
$$
where $b\geq -C$ for some positive 
constant $C$ (since $b$ is smooth). Thanks to the maximum 
principle (\cite[Prop.~4.3]{Chow-Knopf}), 
this implies that $v\geq 0$ on $[0,t_0]\times M$.

Next, suppose that $b$ is irregular with 
$b\in L^p([0,t_0]\times M)$ ($p>d+2$) and 
$b\leq 0$ almost everywhere. Let $v\in W^{1,2,p}_0([0,t_0]\times M)$ 
be the unique solution of the Cauchy problem \eqref{eq:Cauchy-prob3}, 
as supplied by Proposition \ref{prop:well-posed-Cauchy-prob}. 
We would like to conclude that $v$ is nonnegative. 
To this end, approximate $b$ in $L^p([0,t_0]\times M)$ by 
$\seq{b_n}_{n\ge 1}\subset C^\infty([0,t_0]\times M)$ 
with $b_n\leq 0$ for all $n$, and let $v_n$ 
be the corresponding (unique) solution 
in $C^{1,2}([0,t_0]\times M)$ of 
$$
\begin{cases}
\partial_t v_n-\Delta_h v_n+b_n(t,x) v_n
= -b_n(x,t) \;\; \mbox{on $[0,t_0]\times M$},\\
v_n(0,x) = 0 \;\; \mbox{on $M$}.
\end{cases}
$$
Then $v_n\geq 0$. By the a priori estimate
\eqref{eq:Sobol-est-gen-Cauchy-prob}, which now reads 
$$
\norm{v_n}_{W^{1,2,p}} \leq C\norm{b_n}_{L^p},
$$
and the previous discussion, we infer that 
$v_n\stackrel{C^0}{\to} w$ (up to a subsequence), 
for some limit function $0\leq w\in W^{1,2,p}_0([0,t_0]\times M)$ that 
solves \eqref{eq:Cauchy-prob3} with $b\in L^p([0,t_0]\times M)$. 
By uniqueness, we conclude that $v=w\geq 0$. 

To summarize, we have proved that for 
$0\ge b\in L^p([0,t_0]\times M)$ (with $p>d+2$), 
there exists a unique solution 
$0\leq v\in W^{1,2,p}_0([0,t_0]\times M)$ 
of \eqref{eq:Cauchy-prob3}, satisfying
$$
\norm{v}_{W^{1,2,p}} \leq C\norm{b}_{L^p}.
$$

\medskip

We are now in a position to prove the 
main result of this section, namely

\begin{proposition}[test function for duality method]
\label{prop:exist-testfunc}
Suppose $b\in L^p([0,t_0]\times M)$ 
with $p>d+2$ and $b\leq 0$. Then the terminal value problem 
\begin{equation}\label{eq:terminal-value-prob-phi}
\begin{cases}
\partial_t\phi  + \Delta_h\phi  - b(t,x) \phi
= 0 \;\; \mbox{on $[0,t_0]\times M$},\\
\phi(t_0,x) = 1 \;\; \mbox{on $M$},
\end{cases}
\end{equation}
admits a unique solution $\phi \in W^{1,2,p}([0,t_0]\times M)
\cap C^0([0,t_0]\times M)$ with continuous first 
order spatial derivatives. Moreover, $\phi\geq 1$ everywhere 
and the following a priori estimates hold:
\begin{equation}\label{eq:Sobolev-est-testfunc}
\norm{\phi}_{W^{1,2,p}([0,t_0]\times M)} 
\leq T + C(p,M,T) \norm{b}_{L^p([0,t_0]\times M)}
\end{equation}
and (consequently)
\begin{equation}\label{eq:Sobolev-est-testfunc2}
\norm{\phi}_{C^0([0,t_0]\times M)}
+\norm{\nabla \phi}_{C^0([0,t_0]\times M)}
\lesssim_{d,M,p,T} 1+\norm{b}_{L^p([0,t_0]\times M)}.
\end{equation}
\end{proposition}

\begin{proof}
The solution $\phi$ of \eqref{eq:terminal-value-prob-phi} 
is obtained by setting $\phi(t,x):=1 + v(t_0-t,x)$, 
where $v\in W^{1,2,p}_0([0,t_0]\times M)$ is 
the unique solution of the Cauchy problem
$$
\begin{cases}
\partial_t v - \Delta_h v + \tilde{b}(t,x) v
= -\tilde{b}(x,t) \;\; \mbox{on $[0,t_0]\times M$},
\\
v(0,x) = 0 \;\; \mbox{on $M$},
\end{cases}
$$
where $\tilde{b}(t,x):=b(t_0-t,x)$. 
Proposition \ref{prop:anisotropic-embed} 
therefore supplies the existence and uniqueness of $\phi$, 
estimate \eqref{eq:Sobolev-est-testfunc}, and also 
the lower bound $\phi \geq 1$. The final estimate 
\eqref{eq:Sobolev-est-testfunc2} follows from
the anisotropic Sobolev inequality 
(Proposition \ref{prop:anisotropic-embed}) 
and \eqref{eq:Sobolev-est-testfunc}.
\end{proof}

\begin{remark}\label{rem:non-decreasing-rhs}
Observe that the right-hand side 
of \eqref{eq:Sobolev-est-testfunc} is 
non-decreasing in $\norm{b}_{L^p}$, 
a fact that will be exploited 
in Section \ref{sec:proof-main-result}.
\end{remark}



\section{$L^2$ estimate and uniqueness for weak solutions}
\label{sec:L2-est}

The main outcome of this section is an a priori estimate 
that is valid for arbitrary weak $L^2$ solutions
of the SPDE \eqref{eq:target}, with a rough 
velocity field $u$ satisfying in particular 
$\Divh u\in L^p_{t,x}$ for some $p>d+2$. 
The proof relies fundamentally on the 
special noise vector fields $a_i$ constructed in 
Lemma \ref{lem:ellipticity}, the renormalization 
result provided by Theorem \ref{thm:main-resultGK}, and 
a duality method that makes use of the 
test function constructed in 
Proposition \ref{prop:exist-testfunc}. 

\begin{theorem}[$L^2$ estimate and uniqueness]
\label{thm:L2-est-uniq-weak}
Let $\rho$ be an arbitrary weak $L^2$ solution 
of the stochastic continuity equation \eqref{eq:target}, with 
initial datum $\rho_0\in L^2(M)$, velocity vector 
field $u$ satisfying \eqref{eq:u-ass-1}, 
\eqref{eq:u-ass-2}, and \eqref{eq:u-ass-3}, and noise vector 
fields $a_1,\ldots,a_N$ given by 
Lemma \ref{lem:ellipticity}. Then
\begin{equation}\label{eq:improved-apriori-est}
\sup_{0\leq t\leq T}\norm{\rho(t)}^2_{L^2(\Omega\times M)}
\leq  C\norm{\rho_0}^2_{L^2(M)},
\end{equation}
where $C=
C\left(d,M,p,T,a_i,\norm{\Divh u}_{L^p([0,T]\times M)},
\norm{u}_\infty\right)$ is a constant that is 
non-decreasing in $\norm{\Divh u}_{L^p([0,T]\times M)}$ 
and $\norm{u}_\infty$; here, for convenience, we have set 
$\norm{u}_\infty:=\norm{u}_{L^\infty\left([0,T];
\overrightarrow{L^\infty(M)}\right)}$.

Furthermore, weak $L^2$ solutions are 
uniquely determined by their initial data.
\end{theorem}

\begin{proof}
Since $u\in L^1\left([0,T]; \overrightarrow{W^{1,2}(M)}\right)$, 
the weak solution $\rho$ is renormalizable, 
in view of Theorem \ref{thm:main-resultGK}. 
However, Theorem \ref{thm:main-resultGK} asks for
bounded nonlinearities $F$. 
To handle $F(\xi)=\xi^2$, we must 
employ an approximation (truncation) procedure. 

We pick any increasing function 
$\chi\in C^\infty\bigl([0,\infty)\bigr)$ such that
$\chi(\xi)=\xi$ for $\xi\in [0,1]$, $\chi(\xi)=2$ 
for $\xi\geq 2$, $\chi(\xi)\in [1,2]$ for $\xi\in (1,2)$, 
and $A_0:=\sup_{\xi\geq 0}\chi'(\xi)>1$. 
Set $A_1:=\sup_{\xi\geq 0} \abs{\chi''(\xi)}$. 
We define the rescaled function 
$\chi_\mu(\xi)=\mu \,\chi(\xi/\mu)$, for $\mu>0$. 
The relevant approximation of $F(\xi)=\xi^2$ is
$F_\mu(\xi):=\chi_\mu\left(\xi^2\right)$, 
for $\xi\in\R$, $\mu>0$.

Some tedious computations will reveal that
\begin{equation}\label{eq:Fmu-bounds}
\begin{split}
& F_\mu\in C^\infty(\R), 
\quad 
\lim_{\mu\to\infty} F_\mu(\xi)=\xi^2,
\quad
\sup_{\xi\in\R} F_\mu(\xi)\leq 2\mu, 
\quad \sup_{\mu>0} F_\mu(\xi)\leq 2\xi^2,
\\ & 
\sup_{\xi\in\R}\abs{F_\mu'(\xi)}\leq 2\sqrt{2}A_0\sqrt{\mu}, 
\quad 
\sup_{\mu>0} \abs{F_\mu'(\xi)}\leq 2\sqrt{2}A_0\abs{\xi},
\quad 
\lim_{\mu\to\infty} F_\mu'(\xi)=2\xi,
\\ &
\lim_{\mu\to\infty} F_\mu''(\xi)=2,
\quad
\abs{F_\mu''(\xi)}\leq 8A_1 + 2A_0. 
\end{split}
\end{equation}

Furthermore, the function 
$G_{F_\mu}(\xi)=\xi F_\mu'(\xi)-F_\mu(\xi)$ satisfies
\begin{equation*}
\begin{split}
& \sup_{\xi\in\R}\abs{G_{F_\mu}(\xi)}
\leq \left(4A_0 + 2\right)\mu, 
\quad \sup_{\mu>0}\abs{G_{F_\mu}(\xi)}
\leq 2 \left(\sqrt{2}A_0 + 1\right) \xi^2,
\\ & \qquad \text{and} \quad 
\lim_{\mu\to\infty}G_{F_\mu}(\xi)=\xi^2,
\end{split}
\end{equation*}
and the following estimate:
\begin{equation}\label{eq:needed-for-apriori-est}
\abs{G_{F_\mu}(\xi)} 
\leq C_{\chi}F_\mu(\xi), 
\quad \abs{\xi^2F_\mu''(\xi)}\leq C_\chi
\begin{cases}
F_\mu(\xi), 
& \text{for $\abs{\xi}\leq \sqrt{\mu}$}, 
\\ \xi^2, & 
\text{for $\abs{\xi} \in \left[\sqrt{\mu},\sqrt{2\mu}\right]$}, 
\\ O(\mu), & \text{for $\abs{\xi} > \sqrt{2\mu}$},
\end{cases}
\end{equation}
for some constant $C_\chi>0$ independent of $\mu$.

Fix $t_0\in (0,T]$ and consider 
\eqref{eq:renorm-irregular-test} evaluated 
at $t=t_0$ and with $F=F_\mu$. 
Then, in view of the choice of noise vector fields 
$a_i$ (see Lemma \ref{lem:ellipticity}), the 
following equation holds for any 
$\psi\in W^{1,2,p}([0,t_0]\times M)$ 
(as long as$p>d+2$):
\begin{equation}\label{eq:Gr1}
\begin{split}
&\EE\int_M F_\mu(\rho(t_0))\psi(t_0) \, dV_h 
- \EE\int_M F_\mu(\rho_0)\psi(0)\, dV_h 
\\ &\quad 
=\EE\int_0^{t_0}\int_M F_\mu(\rho(s))\partial_t\psi\, dV_h\,ds 
+\EE\int_0^{t_0}\int_M F_\mu(\rho(s))\,u(\psi)\,dV_h\,ds
\\ & \quad \qquad 
+ \EE\int_0^{t_0} \int_M 
F_\mu(\rho(s))\, \Delta_h\psi \, dV_h \,ds
- \frac12 \sum_{i=1}^N \EE\int_0^{t_0} \int_M 
F_\mu(\rho(s))\, \bar{a}_i(\psi) \, dV_h \,ds
\\ & \quad \qquad 
-\EE\int_0^{t_0}\int_M G_{F_\mu}(\rho(s))\Divh u \,\psi\, dV_h\, ds
\\ & \quad \qquad 
-\frac12 \sum_{i=1}^N\EE\int_0^{t_0}\int_M 
\Lambda_i(1)\,G_{F_\mu}(\rho(s))\,\psi\,dV_h\,ds
\\ & \quad \qquad 
+\frac12\sum_{i=1}^N\EE\int_0^{t_0}\int_M 
F_\mu''(\rho(s))\bigl(\rho(s)\Divh a_i\bigr)^2\,\psi\,dV_h\,ds
\\ & \quad \qquad
-\sum_{i=1}^N\EE\int_0^{t_0}\int_M G_{F_\mu}(\rho(s)) 
\bar{a}_i(\psi)\,dV_h \,ds,
\end{split}
\end{equation}
where we have applied Theorem \ref{thm:main-resultGK} 
to \eqref{eq:target-Ito-special-noise} and the  
time-space weak formulation with non-smooth 
test functions $\psi(t,x)$, see Proposition 
\ref{lemma:irregular-test-Gronwall}. 
Let $\phi$ be the unique solution of 
\eqref{eq:terminal-value-prob-phi} with 
$b= -C_\chi\abs{\Divh u}$, where $C_\chi>0$ is the 
constant appearing in \eqref{eq:needed-for-apriori-est}. 
The existence of $\phi$ is guaranteed by 
Proposition \ref{prop:exist-testfunc}. Moreover, 
$\phi$ belongs to $W^{1,2,p}([0,t_0]\times M)
\cap C^0([0,t_0]\times M)$, the estimates 
\eqref{eq:Sobolev-est-testfunc} and 
\eqref{eq:Sobolev-est-testfunc2} hold, and $\phi$ is 
lower bounded by $1$ everywhere in $[0,t_0]\times M$. 
Thanks to Proposition \ref{lemma:irregular-test-Gronwall}, 
we can use $\phi$ as test function in \eqref{eq:Gr1}. 

Making use of \eqref{eq:needed-for-apriori-est}, we obtain
\begin{align*}
& -\EE\int_0^{t_0}\int_M G_{F_\mu}(\rho(s))\Divh u 
\,\phi\, dV_h\, ds 
\leq \EE\int_0^{t_0}\int_M \abs{G_{F_\mu}(\rho(s))}
\abs{\Divh u} \,\phi\, dV_h\, ds 
\\ & \qquad \leq 
\EE\int_0^{t_0}\int_M C_\chi F_\mu(\rho(s))
\abs{\Divh u} \,\phi\, dV_h\, ds
= -\EE\int_0^{t_0}\int_M F_\mu(\rho(s))\, b 
\,\phi\, dV_h\, ds.
\end{align*}
Now, recalling that the test function $\phi$ 
is the unique solution of the PDE 
problem \eqref{eq:terminal-value-prob-phi}, 
the inequality \eqref{eq:Gr1} 
(with $\psi = \phi$) simplifies to
\begin{equation}\label{eq:Gr2}
\begin{split}
&\EE\int_M F_\mu(\rho(t_0))\phi(t_0) \, dV_h 
\leq \EE\int_M F_\mu(\rho_0)\phi(0)\, dV_h  
\\ &\qquad 
+\EE\int_0^{t_0}\int_M F_\mu(\rho(s))\,u(\phi)\,dV_h\,ds
- \frac12 \sum_{i=1}^N \EE\int_0^{t_0} \int_M 
F_\mu(\rho(s))\, \bar{a}_i(\phi) \, dV_h \,ds
\\ & \qquad \quad
-\frac12 \sum_{i=1}^N\EE\int_0^{t_0}\int_M 
\Lambda_i(1)\,G_{F_\mu}(\rho(s))\,\phi\,dV_h\,ds
\\ & \qquad \quad \quad
+\frac12\sum_{i=1}^N\EE\int_0^{t_0}\int_M 
F_\mu''(\rho(s))\bigl(\rho(s)\Divh a_i\bigr)^2\,\phi\,dV_h\,ds
\\ & \qquad\quad\quad\quad 
-\sum_{i=1}^N\EE\int_0^{t_0}\int_M G_{F_\mu}(\rho(s)) 
\bar{a}_i(\phi)\,dV_h \,ds.
\end{split}
\end{equation}

Using the fourth property in \eqref{eq:Fmu-bounds} 
and the estimate \eqref{eq:Sobolev-est-testfunc2} 
satisfied by the solution $\phi$ 
of \eqref{eq:terminal-value-prob-phi} 
with $b= -C_\chi\abs{\Divh u}$, we obtain
\begin{align*}
- \frac12 \sum_{i=1}^N & \EE\int_0^{t_0}
\int_M F_\mu(\rho(s))\, \bar{a}_i(\phi) \, dV_h \,ds 
\\ & \leq C(a_i) \norm{\nabla\phi}_{C^0([0,t_0]\times M)}
\EE\int_0^{t_0} \int_M \rho^2(s) \, dV_h \,ds
\\ & \leq C\left(a_i,\norm{\Divh u}_{L^p([0,T]\times M)}\right) 
\EE\int_0^{t_0} \int_M \rho^2(s) \, dV_h \,ds
\\ & \leq  C\left(a_i,\norm{\Divh u}_{L^p([0,T]\times M)}\right)
\EE\int_0^{t_0} \int_M \rho^2(s)\phi(s) \, dV_h \,ds,
\end{align*}
where we have also exploited that $\phi\ge 1$. 
Observe that the constant $C$ is non-decreasing in 
$\norm{\Divh u}_{L^p([0,T]\times M)}$, 
see Remark \ref{rem:non-decreasing-rhs}, and 
we do not write its dependency on the $d,M,p,T$. 
Similarly, using also 
\eqref{eq:needed-for-apriori-est}, we have
$$
-\sum_{i=1}^N \EE\int_0^{t_0}\int_M G_{F_\mu}(\rho(s)) 
\bar{a}_i(\phi) \,dV_h \,ds 
\leq C\, \EE\int_0^{t_0} \int_M \rho^2(s) \phi(s) 
\, dV_h \,ds,
$$
for a possibly different constant 
$C=C\Big(a_i,\norm{\Divh u}_{L^p([0,T]\times M)}\Big)$, 
still non-decreasing in $\norm{\Divh u}_{L^p([0,T]\times M)}$. 
Similar bounds can be derived for the terms on 
the third and fourth lines of \eqref{eq:Gr2}: 
\begin{align*}
&\frac12\sum_{i=1}^N\EE\int_0^{t_0}\int_M 
F_\mu''(\rho(s))\bigl(\rho(s)\Divh a_i\bigr)^2
\,\phi(s)\,dV_h\,ds
\le C\, \EE\int_0^{t_0} \int_M \rho^2(s) \phi \, dV_h \,ds,
\\ & -\frac12 \sum_{i=1}^N
\EE\int_0^{t_0}\int_M \Lambda_i(1)\,G_{F_\mu}(\rho(s))
\,\phi\,dV_h\,ds
\le C\, \EE\int_0^{t_0} \int_M \rho^2(s)
\phi(s) \, dV_h \,ds.
\end{align*}

Therefore \eqref{eq:Gr2} becomes
\begin{equation*}
\begin{split}
& \EE\int_M F_\mu(\rho(t_0))\phi(t_0) \, dV_h 
\leq \EE\int_M F_\mu(\rho_0)\phi(0)\, dV_h  
\\ & \qquad +\EE\int_0^{t_0}\int_M F_\mu(\rho(s)) 
\,u(\phi)\,dV_h\,ds \\ &\qquad\qquad 
+C\left(a_i,\norm{\Divh u}_{L^p([0,T]\times M)}\right)
\EE\int_0^{t_0} \int_M \rho^2(s)\phi(s) \, dV_h \,ds.
\end{split}
\end{equation*}

Arguing as above, since 
$u\in L^\infty([0,T];\overrightarrow{L^\infty(M)})$, 
\begin{align*}
& \EE\int_0^{t_0}\int_M F_\mu(\rho(s))\,u(\phi)\,dV_h\,ds
\\ & \qquad 
\leq C\left(a_i,\norm{\Divh u}_{L^p([0,T]\times M)},
\norm{u}_\infty\right) 
\EE\int_0^{t_0} \int_M \rho^2(s)\phi(s) 
\, dV_h \,ds,
\end{align*}
where the constant $C$ is non-decreasing 
in $\norm{u}_\infty$ as well.

In conclusion, we have obtained
\begin{equation}\label{eq:Gr4}
\begin{split}
&\EE\int_M F_\mu(\rho(t_0))\phi(t_0) \, dV_h 
\\ & \qquad 
\leq \EE\int_M F_\mu(\rho_0)\phi(0)\, dV_h 
+ C\, \EE\int_0^{t_0} \int_M \rho^2(s)\phi(s) 
\, dV_h \,ds,
\end{split}
\end{equation}
where $C$ depends in particular on $a_1,\ldots,a_N$, 
$\norm{u}_\infty$, and $\norm{\Divh u}_{L^p([0,T]\times M)}$ 
but not on $\mu$; $C$ is non-decreasing in $\norm{u}_\infty$ and 
$\norm{\Divh u}_{L^p([0,T]\times M)}$.   

By the dominated convergence theorem ($\phi$ is 
continuous, $\rho_0\in L^2(M)$), we obtain 
$\EE\int_M F_\mu(\rho_0)\phi(0)\, dV_h\to 
\int_M \rho_0^2\,\phi(0)\, dV_h$ as $\mu\to\infty$. 
On the other hand, by Fatou's lemma, we 
can send $\mu\to\infty$ in the term on the left-hand 
side of \eqref{eq:Gr4}, arriving at
\begin{equation}\label{eq:Gr5}
\EE\int_M \rho^2(t_0)\phi(t_0) \, dV_h 
\leq \int_M \rho_0^2\,\phi(0)\, dV_h 
+ C\, \EE\int_0^{t_0} \int_M 
\rho^2(s)\phi(s) \, dV_h \,ds.
\end{equation}
Since $\phi$ is lower bounded by $1$, we can replace the 
term on the left-hand side by 
$\EE\int_M \rho^2(t_0) \, dV_h$. On the other hand, 
in view of \eqref{eq:Sobolev-est-testfunc2}, we 
can bound (remove) the $\phi$ part from the terms 
on the right-hand side of \eqref{eq:Gr5} by 
$\norm{\phi}_{C^0([0,t_0]\times M)} 
\lesssim 1 + \norm{\Divh u}_{L^p([0,T]\times M)}$,
where $\lesssim$ does not depend on $t_0$. 
As a result, \eqref{eq:Gr5} becomes
$$
\EE\int_M \rho^2(t_0) \, dV_h 
\leq K\int_M \rho_0^2\, dV_h 
+  K\, \EE\int_0^{t_0} \int_M \rho^2(s) \, dV_h \,ds,
$$
where $K$ depends in particular on $a_1,\ldots,a_N$, 
$\norm{u}_\infty$, and $\norm{\Divh u}_{L^p([0,T]\times M)}$, 
still non-decreasing in $\norm{u}_\infty$ and 
$\norm{\Divh u}_{L^p([0,T]\times M)}$. 

Setting $\Phi(t_0):= \EE\int_M \rho^2(t_0)\, dV_h\in [0,\infty)$ 
and $\Phi(0):=C\int_M \rho_0^2\, dV_h\in [0,\infty)$, 
the last inequality reads as
$$
\Phi(t_0)\leq \Phi(0)
+K\int_0^{t_0} \Phi(s) \,ds, 
\qquad 0< t_0\leq T.
$$
The integrability properties of weak solutions 
implies $\Phi\in L^1([0,t_0])$ for any $t_0\leq T$. 
Hence, by Gr\"onwall's inequality,
$$
\Phi(t) \leq \Phi(0) e^{Kt}, 
\qquad t\in [0,T].
$$
This concludes the proof 
of \eqref{eq:improved-apriori-est}, which 
also implies the uniqueness assertion.
\end{proof}

\begin{remark}
Regarding the uniqueness assertion 
in Proposition \ref{thm:L2-est-uniq-weak}, we mention 
that it is possible to prove uniqueness without an 
additional assumption on $\Divh u$. This follows from 
the renormalized formulation \eqref{eq:SPDE-renorm} 
with $F(\cdot)=\abs{\cdot}$, modulo an approximation argument.
Since the existence of weak solutions (which asks 
that $\Divh u\in L^p$) holds in the $L^2$ setting, 
we have chosen not to focus on $L^1$ uniqueness.
\end{remark}


\section{Proof of main result, Theorem \ref{thm:main-result}}
\label{sec:proof-main-result}

We divide the proof of Theorem \ref{thm:main-result}
into four parts (subsections), starting with the procedure
for smoothing the irregular velocity vector field $u$, yielding 
$u_\tau\in C^\infty$ such that $u_\tau\approx u$ 
for $\tau>0$ small. In the second subsection we rely 
on the $L^2$ estimate in Proposition \ref{thm:L2-est-uniq-weak} 
to ensure weak compactness of a sequence $\seq{\rho^\tau}_{\tau>0}$ 
of approximate solutions, obtained by solving the 
the SPDE \eqref{eq:target} with smooth initial datum $\rho_0$ 
and smooth velocity field $u_\tau$. The limit of a 
weakly converging subsequence is easily 
shown to be a solution of the SPDE. 
In the third subsection we remove 
the assumption that $\rho_0$ is smooth. 
Finally, we prove a technical lemma utilized 
in the second subsection.

\subsection{Smoothing of velocity vector field $u$}

We extend the vector field $u$ outside of $[0,T]$ 
by setting $u(t,\cdot)\equiv 0$ for $t<0$ and $t>T$, yielding 
$u\in L^\infty\left(\R;\overrightarrow{L^\infty(M)}\right)$. 

Let $\seq{\mathcal{E}_{\tau}}_{\tau\geq 0}$ denote the 
de Rham-Hodge semigroup on 1-forms, associated 
to the de Rham-Hodge Laplacian on $(M,h)$. 
We refer to Section \ref{sec:appendix} for 
a collection of properties of the 
heat kernel on forms. 

For a.e.~$t\in\R$ and all $\tau>0$, 
$\mathcal{E}_\tau u(t)$ is a smooth vector 
field on $M$ and 
$$
\norm{\mathcal{E}_\tau u(t)}_{\overrightarrow{L^\infty(M)}}
\leq e^{\varepsilon^2\tau}
\norm{u(t)}_{\overrightarrow{L^\infty(M)}},
$$
where $\varepsilon$ is a constant such that 
$\operatorname{Ric}_M\geq -\varepsilon^2 h$. 
By assumption, we clearly have  
$u(t)\in \overrightarrow{L^r(M)}$ 
for a.e.~$t$ and thus
$$
\mathcal{E}_\tau u(t)\totau u(t)
\quad \text{in $\overrightarrow{L^r(M)}$}, 
\qquad r\in [1,\infty),
$$
where the null-set is $r$-independent.

Let $\eta$ be a standard mollifier on $\R$, and set
$$
\eta_\tau(t):=\tau^{-1}\eta\left(\frac{t}{\tau}\right), 
\quad t\in\R.
$$
We now define the following vector field:
$$
u_\tau(t,x):= \int_\R \mathcal{E}_\tau 
u(t',x) \eta_\tau(t-t')\,dt'\in T_xM,
$$
which is well-defined because
\begin{align*} 
\abs{u_\tau(t,x)}_h
& \leq \int_\R \abs{\mathcal{E}_\tau u(t',x)}_h
\eta_\tau(t-t')\,dt' 
\\ & \leq \int_\R
\norm{\mathcal{E}_\tau u(t)}_{\overrightarrow{L^\infty(M)}}
\eta_\tau(t-t')\,dt'
\\ & \leq e^{\varepsilon^2\tau}
\int_\R \norm{u(t')}_{\overrightarrow{L^\infty(M)}}
\eta_\tau(t-t') \,dt'<\infty,
\end{align*}
for any $t\in\R$ and $x\in M$. 
Clearly, $u_\tau:\R\times M\to TM$ is 
smooth in both variables, 
\begin{equation}\label{eq:Linf-utau}
\norm{u_\tau}_{L^\infty\left(\R;\overrightarrow{L^\infty(M)}\right)}
\leq e^{\varepsilon^2\tau} 
\norm{u}_{L^\infty\left(\R;\overrightarrow{L^\infty(M)}\right)},
\end{equation}
and $\supp u_\tau\subset [-1,T+1]\times M$ 
for all $\tau \ll 1$.

For a.e.~$t\in\R$,
\begin{equation}\label{eq:div-commute-semigroup}
\Divh \left(\mathcal{E}_\tau u(t,x)\right) 
= P_\tau\Divh u(t,x), \quad x\in M,
\end{equation}
where $P_\tau$ is the heat kernel on functions. 
Indeed, fixing $\phi\in C^1(M)$, 
we compute (see \cite[eq.~4.3]{Fang-Li-Luo})
\begin{align*}
& \int_M \Divh \left(\mathcal{E}_\tau u(t,x)\right) \phi\,dV_h 
= -\int_M \bigl(\mathcal{E}_\tau u(t,x),\nabla \phi\bigr)_h\,dV_h 
\\ & \qquad
= -\int_M \bigl(u(t,x),\mathcal{E}_\tau\nabla \phi\bigr)_h \,dV_h
= -\int_M \bigl(u(t,x),\nabla P_\tau \phi\bigr)_h\,dV_h
\\ & \qquad = \int_M \Divh u(t,x)P_\tau \phi\,dV_h
= \int_MP_\tau \Divh u(t,x) \phi\,dV_h,
\end{align*}
where we have used the relation \cite{Fang-Li-Luo}
$$
\mathcal{E}_\tau\nabla \phi 
= \nabla P_\tau \phi, 
\quad \phi \in  C^1(M),
$$
and so the identity 
\eqref{eq:div-commute-semigroup} follows.

The next lemma expresses 
$\Divh u_\tau$ in terms of $\Divh u$.

\begin{lemma}[formula for $\Divh u_\tau$]\label{lem:div-utau}
For any $t\in\R$ and $x\in M$,
$$
\Divh u_\tau(t,x):= 
\int_\R \Divh \left(\mathcal{E}_\tau u(t',x)\right)
\eta_\tau(t-t')\,dt',
$$
where $\Divh \left(\mathcal{E}_\tau u(t',x)\right)$ 
can be computed in terms of $\Divh u$ and the heat kernel on 
functions, see \eqref{eq:div-commute-semigroup}.
\end{lemma}

\begin{proof}
Locally expressing $\mathcal{E}_\tau u(t,x)$ as
$$
\mathcal{E}_\tau u(t,x)=
\left(\int_Me(\tau,x,y)_{ij}u^j(t,y)\,dV_h(y)
\right)h^{ik}(x)\, \partial_k,
$$
for a.e.~$t\in\R$, see Section \ref{sec:appendix}, 
we obtain (temporarily dropping Einstein's summation 
convention for $k$)
\begin{align*}
\partial_k\left(\mathcal{E}_\tau u(t,x)\right)^k 
&=\int_M\partial_ke \left(\tau,x,y\right)_{ij}u^j(t,y)
\,dV_h(y)\, h^{ik}(x)
\\ & \qquad
+\int_M e(\tau,x,y)_{ij}u^j(t,y)\,dV_h(y)\,
\partial_kh^{ik}(x),
\end{align*}
and thus 
$$
\abs{\partial_k \left(\mathcal{E}_\tau u(t,x)\right)^k}
\leq C(M,\tau) 
\norm{u}_{\overrightarrow{L^\infty(M)}}.
$$
Therefore we are allowed to interchange 
$\int_\R$ and $\partial_k$ to obtain
$$
\partial_k \left(u_\tau(t,x)\right)^k
= \int_\R \partial_k\left(\mathcal{E}_\tau u(t',x)\right)^k
\eta_\tau(t-t')\,dt'.
$$
From here, recalling the local 
expression for $\Divh$ (see Section \ref{sec:background}), 
it is now immediate to conclude that locally
$$
\Divh u_\tau(t,x)= \int_\R 
\Divh \left(\mathcal{E}_\tau u(t',x)\right)
\eta_\tau(t-t')\,dt'.
$$
\end{proof}

Fix $x\in M$. In view of Lemma \ref{lem:div-utau} 
and basic convolution estimates on $\R$, 
$\norm{\Divh u_\tau(\cdot,x)}_{L^p(\R)}
\leq \norm{\Divh \mathcal{E}_\tau 
u(\cdot,x)}_{L^p(\R)}$ 
for any $\tau>0$, and thus
$$
\norm{\Divh u_\tau}_{L^p(\R\times M)}\leq 
\norm{\Divh \mathcal{E}_\tau u}_{L^p(\R\times M)}.
$$
As a result, via \eqref{eq:div-commute-semigroup}, we obtain
\begin{equation}\label{eq:Lp-bound-divutau}
\begin{split}
\norm{\Divh u_\tau}_{L^p(\R\times M)}
& \leq \norm{P_\tau \Divh u}_{L^p(\R\times M)}
\\ & \leq \norm{\Divh u}_{L^p(\R\times M)}
=\norm{\Divh u}_{L^p([0,T]\times M)}.
\end{split}
\end{equation}

\subsection{Weak compactness of approximate solutions}

Let $\rho^\tau$ be the unique weak $L^2$ 
solution of the SPDE \eqref{eq:target} with 
initial datum $\rho_0\in C^\infty(M)$, noise vector 
fields $a_i$ given by Lemma \ref{lem:ellipticity}, and 
irregular velocity field $u$ (satisfying the assumptions 
of Theorem \ref{thm:main-result}) replaced 
by the smooth vector field $u_\tau$. 

We refer to Propositions \ref{prop:strongsol-is-weaksol} 
and \ref{thm:L2-est-uniq-weak} for the existence, 
uniqueness, and properties of the solution, 
which satisfies the It\^{o} SPDE
\begin{align*}
& d\rho^\tau + \Divh \bigl (\rho^\tau u_\tau \bigr) \,dt 
+ \sum_{i=1}^N \Divh \bigl( \rho^\tau a_i \bigr)\,dW^i(t) 
\\ & \qquad \qquad
- \Delta_h \rho^\tau\, dt -\frac{1}{2}\sum_{i=1}^N 
\Divh \bigl(\rho^\tau \bar{a}_i \bigr)\,dt=0 
\quad \text{weakly in $x$, $\P$-a.s.},
\end{align*}
that is, for any $\psi\in C^\infty(M)$, 
the following equation holds $\P$-a.s.:
\begin{equation}\label{eq:L2-weak-sol-Ito-rho-tau}
\begin{split}
&\int_M \rho^\tau(t)\psi\, dV_h 
= \int_M \rho_0\psi\, dV_h 
+\int_0^t\int_M\rho^\tau(s)\,u_\tau(\psi) \,dV_h\,ds
\\ & \qquad 
+\sum_{i=1}^N\int_0^t\int_M \rho^\tau(s) 
\,a_i(\psi) \, dV_h\, dW^i(s)
+\int_0^t \int_M \rho^\tau(s) 
\,\Delta_h\psi \, dV_h \,ds
\\ & \qquad \qquad
-\frac12\sum_{i=1}^N \int_0^t 
\int_M \rho^\tau(s) \,\bar{a}_i(\psi) \, dV_h \,ds, 
\qquad t\in [0,T].
\end{split}
\end{equation}

In view of \eqref{eq:Linf-utau}, \eqref{eq:Lp-bound-divutau}, 
and \eqref{eq:improved-apriori-est}, recalling 
the ``monotonicity properties" of the constant $C$, 
we obtain the $\tau$-independent $L^2$ estimate
\begin{equation*}
\sup_{0\leq t\leq T}
\norm{\rho^\tau(t)}_{L^2(\Omega\times M)}
\leq  C\left(T,a_i,\norm{\Divh u}_{L^p([0,T]\times M)},
\norm{u}_\infty\right)\norm{\rho_0}_{L^2(M)}.
\end{equation*}
In other words, $\seq{\rho^\tau}_{\tau\in (0,1)}$ is bounded in 
$L^\infty\left([0,T];L^2(\Omega\times M)\right)$. 

Since $\left(L^2(\Omega\times M)\right)^\star$ is 
separable and $\left([0,T],dt\right)$ is a 
finite measure space, we know 
that $L^\infty\left([0,T]; L^2(\Omega\times M)\right)$ 
is the dual of $L^1\left([0,T]; L^2(\Omega\times M)\right)$. 
Therefore, there exist $\seq{\tau_n}_{n\ge 1}
\subset (0,1)$ with $\tau_n\downarrow 0$ and 
$\rho\in L^\infty\left([0,T];L^2(\Omega\times M)\right)$ 
such that
$$
\rho^{\tau_n}\stackrel{\star}{\rightharpoonup} \rho \quad 
\text{in $L^\infty\left([0,T];L^2(\Omega\times M)\right)$},
$$
as $n\to \infty$, which means that
$$
\int_0^T\int_{\Omega}\int_M
\left(\rho^{\tau_n}-\rho\right)
\,\theta\,\, \P\otimes dV_h\otimes dt\ton 0, 
\quad \forall 
\theta \in L^1\left([0,T]; L^2(\Omega\times M)\right).
$$

We follow the arguments in \cite{Pardoux}. 
Fix $\phi\in C^\infty(M)$. The process 
$\int_M\rho^{\tau_n}(t)\phi\,dV_h$ 
is adapted by definition and converges 
weakly in $L^2(\Omega_T)$ to the 
process $\int_M\rho(t)\phi\,dV_h$. 
Since the space of adapted processes is 
a closed subspace of $L^2(\Omega_T)$, it 
is weakly closed, and hence the 
limit process is adapted. 

For the same reason, the processes 
$\int_M\rho^{\tau_n}(t)a_i(\phi)\,dV_h$, $i=1,\ldots,N$, 
are adapted and their It\^o integrals 
are well defined. Since the It\^o 
integral is linear and continuous from the space 
of adapted $L^2(\Omega_T)$ processes to $L^2(\Omega_T)$, 
it is also weakly continuous. As a result,
$$
\int_0^\cdot\int_M\rho^{\tau_n}(s)a_i(\phi)
\,dV_h\, dW^i_s
\weakn
\int_0^\cdot\int_M\rho(s) a_i(\phi) \,dV_h\, dW^i_s 
\quad \text{in $L^2(\Omega_T)$}.
$$   

Exploiting the weak continuity of 
the time-integrals,
$$
\int_0^\cdot \int_M \rho^{\tau_n}(s) 
\Delta_h\phi \, dV_h \,ds
\weakn
\int_0^\cdot \int_M \rho(s) 
\Delta_h\phi \, dV_h \,ds
\quad \text{in $L^2(\Omega_T)$}
$$
and, for $i=1,\ldots, N$, 
$$
\int_0^\cdot \int_M \rho^{\tau_n}(s) 
\bar{a}_i(\phi) \, dV_h \,ds
\weakn
\int_0^\cdot \int_M \rho(s) 
\bar{a}_i(\phi) \, dV_h \,ds
\quad \text{in $L^2(\Omega_T)$}.
$$

It remains to pass to the limit in the 
term involving the velocity field $u_\tau$ 
in \eqref{eq:L2-weak-sol-Ito-rho-tau}. 
The proof of the next lemma 
is postponed to the end of this section.
\begin{lemma}\label{lem:conv-utau}
For any $r\in [1,\infty)$, $u_\tau\to u$ 
in $L^r\left(\R;\overrightarrow{L^r(M)}\right)$ 
as $\tau\downarrow 0$.
\end{lemma}

Lemma \ref{lem:conv-utau} immediately implies 
$$
u_{\tau_n}(\phi) \ton u(\phi) 
\quad \text{in $L^r(\R\times M)$}, 
\qquad r\in [1,\infty).
$$
Using this, the goal is to verify that
\begin{equation}\label{eq:unnamed}
\int_M \rho^{\tau_n} u_{\tau_n}(\phi)\,dV_h 
\weakn \int_M \rho u(\phi)\,dV_h
\quad \text{in $L^2(\Omega_T)$}.
\end{equation}
Fix an arbitrary $\psi\in L^2(\Omega_T)$. Then
\begin{align*}
I(n) & :=\int_{\Omega_T}\left(\int_M 
\rho^{\tau_n}u_{\tau_n}(\phi) \,dV_h \right)
\psi\, \P\otimes ds
-\int_{\Omega_T}\left(\int_M \rho u(\phi)
\,dV_h \right) 
\psi\, \P\otimes ds
\\ & = \int_{\Omega_T}
\left(\int_M \rho^{\tau_n}
\left(u_{\tau_n}(\phi)-u(\phi)\right)
\,dV_h\right) \psi\, \P\otimes ds
\\ & \qquad 
+\int_{\Omega_T} \left(\int_M 
\left(\rho^{\tau_n}-\rho\right)
u(\phi)\,dV_h \right) \psi\, \P\otimes ds
=: I_1(n) + I_2(n).
\end{align*}
By repeated applications of the 
Cauchy-Schwarz inequality, 
\begin{align*}
\bigl| I_1(n)  \bigr| 
& \leq \int_{\Omega_T} \abs{\psi} 
\norm{\rho^{\tau_n}(s)}_{L^2(M)} 
\norm{(u_{\tau_n}-u)(\phi)}_{L^2(M)}
\, \P\otimes ds
\\ & \leq 
\int_0^T \norm{\psi(s)}_{L^2(\Omega)} 
\norm{\rho^{\tau_n}(s)}_{L^2(\Omega\times M)} 
\norm{(u_{\tau_n}-u)(\phi)}_{L^2(M)}\, ds
\\ & \leq 
\norm{\rho^{\tau_n}}_{L^\infty\left([0,T]
;L^2(\Omega\times M)\right)}
\int_0^T \norm{\psi(s)}_{L^2(\Omega)} 
\norm{(u_{\tau_n}-u)(\phi)}_{L^2(M)}\, ds
\\ & \leq C \norm{\psi}_{L^2(\Omega_T)}
\norm{(u_{\tau_n}-u)(\phi)}_{L^2([0,T]\times M)}\ton 0.
\end{align*}

For the $I_2$ term it is 
enough to check that $u(\phi)\psi\in 
L^1\left([0,T];L^2(\Omega\times M)\right)$, because 
in that case we would get $I_2(n)\to 0$ directly 
from the definition of the weak convergence 
$\rho^\tau\weakn \rho$. 
In point of fact, we have 
\begin{align*}
\int_0^T & 
\left(\int_{\Omega\times M}\abs{u(\phi)\psi}^2
\,dV_h\,d \P\right)^{1/2}\, ds
\\ & = \int_0^T\norm{\psi(s)}_{L^2(\Omega)}
\left( \int_M \abs{u(\phi)}^2\,dV_h
\right)^{1/2} \, ds
\\ & \leq \int_0^T \norm{\psi(s)}_{L^2(\Omega)}
\norm{\phi}_{C^1(M)}
\norm{u(s)}_{\overrightarrow{L^\infty(M)}}\, ds
\\ & \leq  \norm{\phi}_{C^1(M)}
\norm{u}_{L^\infty\left([0,T];
\overrightarrow{L^\infty(M)}\right)}
\int_0^T \norm{\psi(s)}_{L^2(\Omega)}\, ds
\\ & \leq  \norm{\phi}_{C^1(M)}
\norm{u}_{L^\infty\left([0,T];
\overrightarrow{L^\infty(M)}\right)}\sqrt{T}
\norm{\psi}_{L^2(\Omega_T)} \,ds<\infty. 
\end{align*}
Therefore $I_2(n)\ton 0$, and thus $I(n)\ton 0$. This 
concludes the proof of \eqref{eq:unnamed}.

We may now pass to the limit in the 
SPDE \eqref{eq:L2-weak-sol-Ito-rho-tau} 
with $\tau=\tau_n$, to conclude that 
$\rho$ satisfies \eqref{eq:L2weak-sol-Ito} 
for a.e.~$(\omega,t)\in\Omega_T$.  Since the 
right-hand-side of \eqref{eq:L2weak-sol-Ito} 
clearly defines a continuous stochastic process, 
the process $\int_M\rho(\cdot,x)\phi(x)\,dV_h(x)$ has 
a continuous modification. In other words, we have 
constructed a weak $L^2$ solution to \eqref{eq:target} 
under the assumption that $\rho_0\in C^\infty(M)$.

\subsection{General initial datum, $\rho_0\in L^2(M)$}

To finish off the proof, we must remove the 
smoothness assumption on the initial datum $\rho_0$. 
We follow the same strategy as above, but this time 
it is simpler since we have to regularize functions 
(not vector fields) defined on the manifold $M$. 

Given $\rho_0\in L^2(M)$, we employ 
the heat semigroup $\seq{P_\tau}_{\tau>0}$ 
on functions to regularize $\rho_0$, see 
Section \ref{sec:appendix} for details. The following 
properties are known: 
\begin{align*}
& P_\tau \rho_0\in C^\infty(M),
\quad \norm{P_\tau \rho_0}_{L^2(M)}
\leq \norm{\rho_0}_{L^2(M)},
\\ & \quad \text{and} \quad
P_\tau \rho_0 \stackrel{L^2(M)}{\longrightarrow} 
\rho_0 \quad \text{as $\tau\downarrow 0$}. 
\end{align*}
According to the previous subsection, there exists a 
unique weak $L^2$ solution $\rho^\tau$ 
of \eqref{eq:target} with initial datum 
$P_\tau\rho_0\in C^\infty(M)$, irregular 
velocity field $u$ satisfying the 
assumptions listed in Theorem \ref{thm:main-result}, 
and noise vector fields $a_i$ given 
by Lemma \ref{lem:ellipticity}. 
As before, Proposition \ref{thm:L2-est-uniq-weak} supplies 
the estimate $\norm{\rho^\tau(t)}_{L^2(\Omega\times M)}
\leq  C_T\norm{\rho_0}_{L^2(M)}$ for all $t\in [0,T]$, 
where the constant $C_T$ is independent of $\tau$. 
This implies that $u^\tau$ is 
weakly compact, i.e., there exist a subsequence 
$\seq{\tau_n}_{n\ge 1}\subset (0,1)$ 
with $\tau_n\ton 0$ and a limit $\rho\in L^\infty\left([0,T];
L^2(\Omega\times M)\right)$ such that
$$
\rho^{\tau_n}\stackrel{\star}{\rightharpoonup} 
\rho \quad\text{in $L^\infty\left([0,T]; 
L^2(\Omega\times M)\right)$}.
$$
For any $\phi\in C^\infty(M)$, we have trivially that
$$
\int_M P_{\tau_n}\rho_0\, \phi\, dV_h  
\weakn \int_M \rho_0\,\phi\, dV_h 
\quad\text{in $L^2(\Omega_T)$}.
$$
The limit of the remaining terms 
in \eqref{eq:L2-weak-sol-Ito-rho-tau} 
can be computed as before, which in the end  
leads to the conclusion that $\rho$ is 
a weak $L^2$ solution of \eqref{eq:target}. 

\subsection{Proof of Lemma \ref{lem:conv-utau}}

To conclude the proof of Theorem \ref{thm:main-result}, 
we need to verify the validity of Lemma \ref{lem:conv-utau}. 
Define for convenience
$$
\mathcal{J}_\tau(t,x):=\int_\R u(t',x)\eta_\tau(t-t')\,dt', 
\quad t\in\R,\,\, x\in M.
$$
We have
$$
\abs{u_\tau(t,x)-\mathcal{J}_\tau(t,x)}_h 
\leq \int_\R\abs{\mathcal{E}_\tau 
u(t',x)-u(t',x)}_h\eta_\tau(t-t')\,dt'.
$$
By basic convolution estimates on $\R$, 
for any $r\in [1,\infty)$,
$$
\norm{\,\abs{u_\tau(\cdot,x)
-\mathcal{J}_\tau(\cdot,x)}_h\,}_{L^r(\R)}
\leq 
\norm{\,\abs{\mathcal{E}_\tau u(\cdot,x)
-u(\cdot,x)}_h\,}_{L^r(\R)},
\quad x\in M,
$$
where $\norm{\, \abs{\cdot}_h \,}_{L^r(\R)}^r
=\int_{\R} \abs{\cdot}_h^r\,dt\,$. Thus,
\begin{align*}
&\norm{u_\tau
-\mathcal{J}_\tau}_{L^r\left(\R;\overrightarrow{L^r(M)}\right)}
\\ & \qquad \leq \norm{\mathcal{E}_\tau u
-u}_{L^r\left(\R;\overrightarrow{L^r(M)}\right)}
= \left(\int_\R\norm{\mathcal{E}_\tau u(t,\cdot)
-u(t,\cdot)}_{\overrightarrow{L^r(M)}}^r 
\,dt \right)^{1/r}.
\end{align*}
Observe that the integrand in the $dt$-integral 
converges to zero as $\tau\downarrow 0$ for a.e.~$t\in\R$. 
Furthermore, see Section \ref{sec:appendix},
$$
\norm{\mathcal{E}_\tau u(t,\cdot)
-u(t,\cdot)}_{\overrightarrow{L^r(M)}}
\leq \bigl(\exp\left(\varepsilon^2\abs{1-2/r}\tau\right) 
+ 1\bigr) \norm{u(t,\cdot)}_{\overrightarrow{L^r(M)}},
$$
which is integrable on $\R$ by assumption on $u$ 
(here $-\varepsilon$ is a lower bound of 
the Ricci tensor on $M$). Therefore, by means of the 
dominated convergence theorem, we conclude 
that $u_\tau - \mathcal{J}_\tau\to 0$ 
in $L^r\left(\R;\overrightarrow{L^r(M)}\right)$ 
as $\tau\downarrow 0$.

Hence, with an error term 
$o(1)\to 0$ as $\tau\downarrow 0$,
$$
u_\tau -u=\mathcal{J}_\tau - u + o(1),
$$
so it remains to verify that $\mathcal{J}_\tau - u$ converges 
to zero in $L^r_t\overrightarrow{L^r_x}$. 
Locally we have $\abs{\mathcal{J}_\tau(t,x) - u(t,x)}_h
\leq C(M,h) \abs{\mathcal{J}_\tau(t,x) 
- u(t,x)}_{\operatorname{eucl}}$. Since the right-hand 
side converges to zero in $L^r(\R)$ for all $x\in M$, it 
follows that the same holds for the left-hand side. 
We have 
\begin{align*}
\int_\R\int_M & 
\abs{\mathcal{J}_\tau(t,x) - u(t,x)}_h^rdV_h(x)\,dt 
= \int_M\int_\R \abs{\mathcal{J}_\tau(t,x) - u(t,x)}_h^r 
\,dt \,dV_h(x)
\\ & = \sum_{\kappa}\int_M \alpha_\kappa(z)
\left(\int_\R \abs{\mathcal{J}_\tau(t,z)
- u(t,z)}_h^r \,dt\right)\,\abs{h_\kappa(z)}^{1/2}\, dz,
\end{align*}
where $(\alpha_\kappa)_\kappa$ is an arbitrary 
smooth partition of unity. Arguing as we did above, 
$\norm{\, \abs{\mathcal{J}_\tau(\cdot,x)
-u(\cdot,x)}_h \, }_{L^r(\R)}^r
\leq 2^r\norm{\, \abs{u(\cdot,x)}_h \, }_{L^r(\R)}^r$ 
for any $x\in M$, and hence, by means of the 
dominated convergence theorem, 
$\mathcal{J}_\tau - u\to 0$ in 
$L^r\left(\R;\overrightarrow{L^r(M)}\right)$. 
This concludes the proof of Lemma \ref{lem:div-utau}.


\section{Appendix}
\label{sec:appendix}


\subsection{Heat kernel on functions}
\label{subsec:heat-kernel-func}
We collect here some relevant properties 
of the heat kernel $H$ on $(M,h)$, that is, 
the fundamental solution of the heat operator
$$
L=\partial_t -\Delta_h.    
$$
\begin{enumerate}
\item the mapping $(x,y,t)\mapsto H(x,y,t)$ belongs 
to $C^\infty(M\times M\times (0,\infty))$, is 
symmetric in $x$ and $y$ for any $t>0$, and is positive. 

\item For any function $w\in L^r(M)$, 
$r\in [1,\infty]$, setting
\begin{equation}\label{eq:heat-kernel-prop}
P_t w(x):=\int_M H(x,y,t)w(y)\,dV_h(y), 
\qquad x\in M,  \,\, t>0,
\end{equation}
we have $P_t w\in C^\infty(M)$. Moreover,
$$
\norm{P_tw}_{L^r(M)}
\leq \norm{w}_{L^r(M)}, 
\qquad t> 0, 
$$
and, for any finite $r\ge 1$,
\begin{equation*}
P_t w \stackrel{L^r(M)}{\longrightarrow} w 
\quad \text{as $t\to 0^+$}.
\end{equation*}
\end{enumerate}

For proofs of these basic results, see \cite{Grig}.


\subsection{Heat kernel on forms}
\label{subsec:heat-kernel-forms}

During the proof of Theorem \ref{thm:main-result}, 
we also make use of the heat kernel on forms. 
We recall here its most salient 
properties without proofs, referring 
to \cite{Elworthy,Bakry,deRham,GK0} for details. 
Firstly, we define the space $\Ltwo$ as the closure 
of the space of smooth 1-forms on $M$ with respect to the norm
$$
\left(\int_M \abs{\zeta}_h^2\,dV_h\right)^{1/2},
\qquad \text{where $\abs{\zeta}_h^2
=h^{ab}\zeta_a\zeta_b\,$ locally.}    
$$

Denote by $\seq{\mathcal{E}_\tau}_{\tau\geq 0}$ 
the de Rham-Hodge semigroup on 1-forms, 
associated to the de Rham-Hodge Laplacian, 
which by elliptic regularity has a kernel 
$e(\tau,\cdot,\cdot)$. More precisely, for any 
$\tau>0$, $e(\tau,\cdot,\cdot)$ is a double 
form on $M\times M$, such that for any 1-form 
$\zeta\in \Ltwo$ and any $P\in M$,
$$
\left(\mathcal{E}_\tau\zeta\right)(P)
= \int_M e(\tau,P,Q)\wedge\star_Q\,\zeta(Q), 
$$
where $\star$ is the Hodge star operator, 
$Q$ is a point in $M$, and $\wedge$ is the 
wedge product between forms. Concretely, in a 
coordinate patch $\left(U,(x^i)\right)$ around $P$ 
and in a coordinate patch $\left(U',(y^j)\right)$ around $Q$, 
if we write the double form $e(\tau,\cdot,\cdot)$ as
$$
e(\tau,x,y)=\left(e(\tau,x,y)_{ij}\,dx^i\right)\, dy^j
$$
and $\zeta$ as $\zeta(y)=\zeta_k(y)\, dy^k$, then 
the above integral becomes
$$
(\mathcal{E}_\tau\zeta)(x)= \left(\int_M e(\tau,x,y)_{ij}
\,h^{jk}(y)\,\zeta_k(y) \,dV_h(y)\right)\, dx^i. 
$$
For a vector field $V$, we denote by $V^\flat$ 
the 1-form obtained by lowering an index 
via the metric $h$; analogously, 
for a 1-form $\zeta$, we denote 
by $\zeta^\sharp$ the vector field obtained 
by raising an index via the metric.

We define for a vector field $V$ 
the following quantity
$$
\mathcal{E}_\tau V := \left( 
(\mathcal{E}_\tau V^\flat) \right)^\sharp. 
$$ 

Let $\varepsilon\geq 0$ be a constant 
such that $\operatorname{Ric}_M\geq 
-\varepsilon^2 h$, where $\operatorname{Ric}_M$ 
denotes the Ricci tensor of $(M,h)$ 
(the constant $\varepsilon$ clearly 
exists because $M$ is compact). 
We have the following remarkable properties: 
for any $V\in \overrightarrow{L^p(M)}$, $p\in [1,\infty]$,
\begin{itemize}

\item $\mathcal{E}_\tau V$ is a smooth 
vector field, for any $\tau>0$,

\item $\mathcal{E}_\tau V\to V$ 
in $\overrightarrow{L^p(M)}$ as 
$\tau\downarrow 0$, for any finite $p$,

\item $\norm{\mathcal{E}_\tau V}_{\overrightarrow{L^p(M)}}
\leq e^{\varepsilon^2\abs{1-\frac{2}{p}}\tau}\,
\norm{V}_{\overrightarrow{L^p(M)}}$, 
for any $\tau\geq 0$ \cite{Bakry}. 
\end{itemize}

Furthermore, in analogy with \eqref{eq:heat-kernel-prop}, 
the following local expression holds:
$$
(\mathcal{E}_\tau V)(x)= 
\left(\int_M e(\tau,x,y)_{ij} 
\,V^{j}(y) \,dV_h(y)
\right)h^{ik}(x) \, \partial_k. 
$$
 
Finally, one can show that (see \cite{GK0} for details)
\begin{equation*}
\begin{split}
\Div \,\mathcal{E}_\tau V(x) 
& = \int_M  \partial_k e(\tau,x,y)_{ij}\,V^{j}(y) 
\,dV_h(y)\,h^{ik}(x)
\\ & \qquad 
+\int_M e(\tau,x,y)_{ij}\,V^{j}(y)
\,dV_h(y)\,\partial_k h^{ik}(x)
\\ & \qquad\qquad
+\Gamma^\rho_{\rho k}(x)
\int_M e(\tau,x,y)_{ij}\,V^{j}(y) 
\,dV_h(y)\, h^{ik}(x), 
\end{split}
\end{equation*}
in local coordinates $x$ (differentiation is 
carried out in $x$).


\subsection{Proof of Proposition \ref{prop:anisotropic-embed}}
\label{subsec:proof-anisotropic-Sobolev}

Let $\seq{G_i}_{i=1}^R$ be a finite covering of $M$ 
and $\seq{(G_i,\phi_i)}_{i=1}^R$ the corresponding charts. 
Without loss of generality, we may assume that 
$\phi_i(G_i)=B$ for all $i$, where $B$ is 
the unit ball in $\R^d$. Let $\seq{\alpha_i}_{i=1}^R$ 
be a smooth partition of unity subordinate 
to $\{G_i\}_{i=1}^R$. On $\supp \alpha_i$ 
the metric tensor $h$ and its derivatives 
of all orders are bounded in the system 
of coordinates corresponding to the chart 
$(G_i,\phi_i)$. Define, for $i=1,\ldots,R$,
$$
\psi_i:[0,T]\times G_i\to [0,T]\times B, 
\qquad (t,P)\mapsto (t,\phi_i(P)),
$$
which is a finite smooth atlas for $[0,T]\times M$. 
Observe that $[0,T]\times G_i$ is 
diffeomorphic to $[0,T]\times B$. 
Moreover, $\tilde{\alpha}_i:[0,T]\times M\to [0,1]$, 
$\tilde{\alpha}_i(t,P):=\alpha_i(P)$ 
is a smooth partition of unity 
subordinate to $\seq{[0,T]\times G_i}_{i=1}^R$.

Let $w\in W^{1,2,p}([0,T]\times M)$. 
Then clearly, in view of the discussion above,
\begin{align*}
	w \in W^{1,2,p}([0,T]\times M) 
	& \Longleftrightarrow 
	\tilde{\alpha}_iw \in W^{1,2,p}([0,T]\times M), 
	\quad \forall i
	\\ & \Longleftrightarrow 
	\left(\tilde{\alpha}_iw \right)\circ \psi^{-1}_i 
	\in W^{1,2,p}([0,T]\times B), \quad \forall i,
\end{align*}
where $W^{1,2,p}([0,T]\times B)$ denotes 
the more familiar Euclidean anisotropic Sobolev 
space \cite{Besov}, which can be defined 
similarly via \eqref{eq:anisotropic-norm} 
with $dV_h=dx$ and $\nabla^k=\nabla^k_{\operatorname{eucl}}$.

For this space we have the compact embedding
$$
W^{1,2,p}([0,T]\times B) 
\subset \subset 
C^{0,1-\frac{1+d}{p}}
\left([0,T]\times \bar{B}\right)
$$
and 
$$
\partial_{x_j} \left(\tilde{\alpha_i}w \right)
\circ \psi^{-1}_i\in 
C^{0,1-\frac{1+d}{p}}\left([0,T]\times \bar{B}\right), 
\quad j=1,\ldots,d,
$$
provided $p>d+2$, see \cite{Rabier} for example. 
In particular, for all $i$,
\begin{align*}
\norm{(\tilde{\alpha_i}w)\circ 
\psi^{-1}_i}_{C^0\left([0,T]\times \bar{B}\right)} 
+ & \norm{\nabla_{\operatorname{eucl}}
\left(\tilde{\alpha_i}w\right)
\circ \psi^{-1}_i}_{C^0\left([0,T]\times \bar{B}\right)}
\\ & \leq C(p,d,B) \norm{\left(\tilde{\alpha_i}w\right)
\circ \psi^{-1}_i}_{W^{1,2,p}([0,T]\times B)}.
\end{align*}
Exploiting the boundedness 
of the metric tensor, we get
\begin{align*}
&\norm{\tilde{\alpha_i}w}_{C^0([0,T]\times M)} 
+\norm{\nabla(\tilde{\alpha_i}w)}_{C^0([0,T]\times M)}
\\ & \quad =\norm{\tilde{\alpha_i}w}_{C^0([0,T]\times G_i)} 
+\norm{\nabla \left(\tilde{\alpha_i}w\right)}_{C^0([0,T]\times G_i)}
\\ & \quad \leq \norm{\left(\tilde{\alpha_i}w\right)
\circ \psi^{-1}_i}_{C^0\left([0,T]\times \bar{B}\right)} 
+C_i\norm{\nabla_{\operatorname{eucl}}
\left(\tilde{\alpha_i}w\right)
\circ \psi^{-1}_i}_{C^0\left([0,T]\times \bar{B}\right)}
\\ & \quad \leq C(p,d,B,i) \norm{\left(\tilde{\alpha_i}w\right)
\circ \psi^{-1}_i}_{W^{1,2,p}([0,T]\times B)}
\\ & \quad \leq C'(p,d,B,i) 
\norm{\tilde{\alpha_i}w}_{W^{1,2,p}([0,T]\times M)}.
\end{align*}
Therefore, by the triangle inequality 
and summing over $i$,
\begin{align*}
&\norm{w}_{C^0([0,T]\times M)}
+\norm{\nabla w}_{C^0([0,T]\times M)}
\\ & \quad \leq C(p,d,M)\sum_{i=1}^R
\norm{\tilde{\alpha_i}w}_{W^{1,2,p}([0,T]\times M)}
\leq C'(p,d,M)\norm{w}_{W^{1,2,p}([0,T]\times M)},
\end{align*}
where in the last passage we have 
used the fact that the derivatives 
of $\tilde{\alpha}_i$ are bounded. 
The compactness of the embedding 
is now evident.


\subsection{An auxiliary result}
We now prove a useful result about the extension of smooth 
functions, which is used during the 
proof of Proposition \ref{lemma:irregular-test-Gronwall}.
\begin{proposition}[extension of $C^\infty$ functions]
\label{prop:Seeley}
Let $0<S<T$ and consider $w\in C^{\infty}([0,S]\times M)$. 
Then we can extend $w$ to a function 
$v\in C^{\infty}([0,T]\times M)$.
\end{proposition}

\begin{proof}
Let $\seq{G_i}_{i=1}^R$ be a finite covering of $M$ 
and $\seq{(G_i,\phi_i)}_{i=1}^R$ the corresponding charts. 
Without loss of generality, we may assume 
that $\phi_i(G_i)=B$ for all $i$, where $B$ is 
the unit ball in $\R^d$. Let $\seq{\alpha_i}_{i=1}^R$ 
be a squared smooth partition of unity subordinate 
to $\seq{G_i}_{i=1}^R$, such that 
$\sum_{i=1}^R \alpha_i^2=1$. Define, for $i=1,\ldots,R$,
$$
\psi_i:[0,T]\times G_i\to [0,T]\times B, 
\qquad (t,P)\mapsto (t,\phi_i(P)),
$$
which is a finite smooth atlas for $[0,T]\times M$. 
Observe that $[0,T]\times G_i$ is 
diffeomorphic to $[0,T]\times B$. Besides, 
$\tilde{\alpha}_i:[0,T]\times M\to [0,1]$, 
$\tilde{\alpha}_i(t,P):=\alpha_i(P)$ is 
a squared smooth partition of unity subordinate to 
$\seq{[0,T]\times G_i}_{i=1}^R$.

Given $w\in C^{\infty}([0,S]\times M)$, we define 
$\tilde{w}_i\in C^{\infty}([0,S]\times \R^d)$, 
$i=1,\ldots ,R$, by
$$
\tilde{w}_i(t,x):=
\begin{cases}
\left(\tilde{\alpha}_i w\right)\circ \psi_i^{-1},
& \text{for $(t,x)\in [0,S]\times B$},
\\ 0, & \text{otherwise.}
\end{cases}
$$
Observe that for any $t\in [0,S]$, 
$\supp \tilde{w}_i(t,\cdot)\subset 
\supp \alpha_i\circ\phi_i^{-1} $. 
Seeley's extension theorem \cite{Seeley} supplies 
an extension operator
$$
\mathcal{E}:C^{\infty}([0,S]\times \R^d) 
\to C^{\infty}(\R\times \R^d).
$$
Thanks to this, we can build an extension 
$\mathcal{E}\tilde{w}_i$ of $\tilde{w}_i$ 
in $C^{\infty}(\R\times \R^d)$. Set
$$
\bar{w}_i:=\left(\alpha_i\circ\phi_i^{-1}\right) 
\mathcal{E}\tilde{w}_i 
\in C^{\infty}(\R\times \R^d),
$$
and notice that for any $t\in \R$, 
$\supp \bar{w}_i(t,\cdot)\subset 
\supp \alpha_i\circ\phi_i^{-1}$. 

We may lift this function to $M$ by setting
$$
w_i(t,P):=
\begin{cases}
\bar{w}_i(t,\phi(P)) , 
& \text{for $(t,P)\in \R\times \supp \alpha_i$},
\\ 0, & \text{otherwise}.
\end{cases}
$$
Clearly, $w_i\in C^{\infty}(\R\times M)$ 
and for $t\in [0,S]$ we have
\begin{equation*}
w_i(t,P)
=\begin{cases}
\alpha_i^2(P)w(t,P), 
&\text{for $P\in \supp \alpha_i$}
\\ 0, & \text{otherwise.}
\end{cases}
\end{equation*}
Setting $\displaystyle v:=\sum_{i=1}^R w_i\in 
C^{\infty}(\R\times M)\subset C^{\infty}([0,T]\times M)$ 
we have
\begin{equation*}
    v(t,P)=\sum_{i: P\in \supp \alpha_i} 
    \alpha_i^2(P) w(t,P) = w(t,P), 
    \qquad (t,P)\in[0,S]\times M,
\end{equation*}
and thus the desired extension is established.
\end{proof}



\begin{thebibliography}{10}

\bibitem{Ambrosio:2004aa}
L.~Ambrosio.
\newblock Transport equation and {C}auchy problem for {$BV$} vector fields.
\newblock {\em Invent. Math.}, 158(2):227--260, 2004.

\bibitem{Amorim:2005aa}
P.~Amorim, M.~Ben-Artzi, and P.~G. LeFloch.
\newblock Hyperbolic conservation laws on manifolds: total variation estimates
  and the finite volume method.
\newblock {\em Methods Appl. Anal.}, 12(3):291--323, 2005.

\bibitem{Attanasio:2011fj}
S.~Attanasio and F.~Flandoli.
\newblock Renormalized solutions for stochastic transport equations 
and the regularization by bilinear multiplication noise.
\newblock {\em Comm. Partial Differential Equations}, 36(8):1455--1474, 2011.

\bibitem{Aubin}
T.~Aubin.
\newblock {\em Some nonlinear problems in {R}iemannian geometry}.
\newblock Springer Monographs in Mathematics. Springer-Verlag, Berlin, 1998.

\bibitem{Bakry}
D.~Bakry.
\newblock \'etude des transformations de {R}iesz dans 
les vari\'et\'es riemanniennes \`a courbure de {R}icci minor\'ee.
\newblock In {\em S\'eminaire de {P}robabilit\'es, {XXI}}, volume 1247 
of {\em Lecture Notes in Math.}, pages 137--172. Springer, Berlin, 1987.

\bibitem{Beck:2019}
L.~Beck, F.~Flandoli, M.~Gubinelli, and M.~Maurelli.
\newblock Stochastic {ODE}s and stochastic linear {PDE}s 
with critical drift: regularity, duality and uniqueness.
\newblock {\em Electron. J. Probab.}, 24:Paper No. 136, 72, 2019.

\bibitem{Ben-Artzi:2007aa}
M.~Ben-Artzi and P.~G. LeFloch.
\newblock Well-posedness theory for geometry-compatible 
hyperbolic conservation laws on manifolds.
\newblock {\em Ann. Inst. H. Poincar\'e Anal. Non Lin\'eaire}, 
24(6):989--1008, 2007.

\bibitem{Besov}
O.V.~Besov, V.P.~Il'in. S.M.~Nikol'skij. 
\newblock {\em Integral Representations of Functions and Imbedding Theorems, Vol 1}.
\newblock Wiley, New York, 1978.

\bibitem{Chow:2015aa}
P.-L. Chow.
\newblock {\em Stochastic partial differential equations}.
\newblock Advances in Applied Mathematics. CRC Press, Boca Raton, FL, second
  edition, 2015.

\bibitem{Chow-Knopf}
B.~Chow, D.~Knopf.
\newblock {\em The Ricci Flow: An Itroduction}.
\newblock American Mathematical Society, 2004.

\bibitem{Davies}
E.~B. Davies.
\newblock Pointwise bounds on the space and time derivatives of heat kernels.
\newblock {\em J. Operator Theory}, 21(2):367--378, 1989.

\bibitem{deRham}
G.~de~Rham.
\newblock {\em Differentiable manifolds}, volume 266 of {\em Grundlehren 
der Mathematischen Wissenschaften}.
\newblock Springer-Verlag, Berlin, 1984.

\bibitem{DL89}
R.~J. DiPerna and P.-L. Lions.
\newblock Ordinary differential equations, transport theory 
and {S}obolev spaces.
\newblock {\em Invent. Math.}, 98(3):511--547, 1989.

\bibitem{Dumas:1994aa}
H.~S. Dumas, F.~Golse, and P.~Lochak.
\newblock Multiphase averaging for generalized flows on manifolds.
\newblock {\em Ergodic Theory Dynam. Systems}, 14(1):53--67, 1994.

\bibitem{Elliott:2012aa}
C.~M. Elliott, M.~Hairer, and M.~R. Scott.
\newblock Stochastic partial differential equations 
on evolving surfaces and evolving {Riemannian} manifolds.
\newblock arXiv:1208.5958.

\bibitem{Elworthy}
D.~Elworthy.
\newblock Geometric aspects of diffusions on manifolds.
\newblock In {\em \'Ecole d'\'Et\'e de {P}robabilit\'es de 
{S}aint-{F}lour {XV}--{XVII}, 1985--87}, 
volume 1362 of {\em Lecture Notes in Math.}, 
pages 277--425. Springer, Berlin, 1988.

\bibitem{Fang-Li-Luo}
S.~Fang, H.~Li, and D.~Luo.
\newblock Heat semi-group and generalized flows 
on complete {R}iemannian manifolds.
\newblock {\em Bull. Sci. Math.}, 135(6-7):565--600, 2011.

\bibitem{Flandoli:2011vn}
F.~Flandoli.
\newblock {\em Random perturbation of {PDE}s 
and fluid dynamic models}, volume 2015 of {\em Lecture Notes in Mathematics}.
\newblock Springer, Heidelberg, 2011.
\newblock Lectures from the 40th Probability Summer 
School held in Saint-Flour,  2010.

\bibitem{Flandoli-Gubinelli-Priola}
F.~Flandoli, M.~Gubinelli, and E.~Priola.
\newblock Well-posedness of the transport equation by stochastic perturbation.
\newblock {\em Invent. Math.}, 180(1):1--53, 2010.

\bibitem{Friedman}
A.~Friedman.
\newblock {\em Stochastic Differential Equations and Applications}.
\newblock Dover, 2006.

\bibitem{GK0}
L.~Galimberti and K.~H. Karlsen.
\newblock Well-posedness theory for stochastically forced 
conservation laws on Riemannian manifolds.
\newblock {\em J. Hyperbolic Differ. Equ.}, 16(3):519--593, 2019.

\bibitem{GK}
L.~Galimberti and K.~H. Karlsen.
\newblock Renormalization of stochastic continuity equations 
on {R}iemannian manifolds.
\newblock {\em Stochastic Processes and their Applications}, 
142:195--244, 2021.

\bibitem{Gess:2018aa}
B.~Gess and M.~Maurelli.
\newblock Well-posedness by noise for scalar conservation laws.
\newblock {\em Comm. Partial Differential Equations}, 
43(12):1702--1736, 2018.

\bibitem{Gess:2019aa}
B.~Gess and S.~Smith.
\newblock Stochastic continuity equations with conservative noise.
\newblock {\em J. Math. Pures Appl. (9)}, 128:225--263, 2019.

\bibitem{Greene:1979aa}
R.~E. Greene and H.~Wu.
\newblock {$C^{\infty }$} approximations of convex, subharmonic, and
  plurisubharmonic functions.
\newblock {\em Ann. Sci. \'{E}cole Norm. Sup. (4)}, 12(1):47--84, 1979.

\bibitem{Grig}
A.~Grigor'yan.
\newblock {\em Heat kernel and analysis on manifolds}, 
volume~47 of {\em AMS/IP 
Studies in Advanced Mathematics}.
\newblock American Mathematical Society, 2009.

\bibitem{Gyongy:1993aa}
I.~Gy\"{o}ngy.
\newblock Stochastic partial differential equations on manifolds. {I}.
\newblock {\em Potential Anal.}, 2(2):101--113, 1993.

\bibitem{Gyongy:1997aa}
I.~Gy\"{o}ngy.
\newblock Stochastic partial differential equations on manifolds. {II}.
  {N}onlinear filtering.
\newblock {\em Potential Anal.}, 6(1):39--56, 1997.

\bibitem{Holden:2020aa}
H.~Holden, K.~H. Karlsen, and P.~H. Pang.
\newblock The {H}unter--{S}axton equation with noise.
\newblock {\em J. Differential Equations}, 270:725--786, 2021.

\bibitem{Kunita82}
H.~Kunita.
\newblock First order stochastic partial differential equations
\newblock {\em Stochastic Analysis}, 
1982 North-Holland Math. Library vol. 32, 249-269.

\bibitem{Kunita}
H.~Kunita.
\newblock {\em Stochastic flows and stochastic differential equations}, 
volume~24 of {\em Cambridge studies in advanced mathematics}.
\newblock Cambridge University Press, 1990.

\bibitem{LeeRiemann}
J.~M. Lee.
\newblock {\em Riemannian Manifolds}, 
volume 176 of {\em Graduate Texts in Mathematics}.
\newblock Springer-Verlag, New York, 1997.

\bibitem{Lions:NSI}
P.-L. Lions.
\newblock {\em Mathematical topics in fluid mechanics. 
{V}ol. 1: Incompressible models}.
\newblock Oxford University Press, New York, 1996.

\bibitem{Lions:NSII}
P.-L. Lions.
\newblock {\em Mathematical topics in fluid mechanics. 
{V}ol. 2: Compressible models.}
\newblock Oxford University Press, New York, 1998.

\bibitem{Neves:2015aa}
W.~Neves, C.~Olivera.
\newblock Wellposedness for stochastic continuity equations 
with Ladyzhenskaya-Prodi-Serrin condition.
\newblock {\em NoDEA Nonlinear Differential Equations Appl.}, 
22(5):1247--1258, 2019.

\bibitem{Neves:2016aa}
W.~Neves and C.~Olivera.
\newblock Stochastic continuity equations--a general uniqueness result.
\newblock {\em Bulletin of the Brazilian 
Mathematical Society, New Series}, 47:631--639, 2016.

\bibitem{Pardoux}
E.~Pardoux.
\newblock {\em Equations aux d\'eriv\'ees partielles stochastiques 
non lin\'eaires monotones. Etude de solutions fortes de type Ito}, PhD Thesis.
\newblock Universit\'e Paris Sud (1975).

\bibitem{Protter}
P.~Protter.
\newblock {\em Stochastic integration and differential equations}, volume~21 of
  {\em Applications of Mathematics (New York)}.
\newblock Springer-Verlag, Berlin, 1990.

\bibitem{Punshon-Smith:2017aa}
S.~Punshon-Smith.
\newblock Renormalized solutions to stochastic continuity equations with rough
  coefficients.
\newblock arXiv 1710.06041.

\bibitem{Punshon-Smith:2018aa}
S.~Punshon-Smith and S.~Smith.
\newblock On the {B}oltzmann equation with stochastic 
kinetic transport: global existence of 
renormalized martingale solutions.
\newblock {\em Arch. Ration. Mech. Anal.}, 229(2):627--708, 2018.

\bibitem{Rabier}
P.~J. Rabier.
\newblock Vector-valued {M}orrey's embedding theorem and {H}\"{o}lder
  continuity in parabolic problems.
\newblock {\em Electron. J. Differential Equations}, pages No. 10, 10, 2011.

\bibitem{Revuz:1999aa}
D.~Revuz and M.~Yor.
\newblock {\em Continuous martingales and {B}rownian motion}, volume 293 of
  {\em Grundlehren der Mathematischen Wissenschaften [Fundamental Principles of
  Mathematical Sciences]}.
\newblock Springer-Verlag, Berlin, third edition, 1999.

\bibitem{Rossmanith:2004aa}
J.~A. Rossmanith, D.~S. Bale, and R.~J. LeVeque.
\newblock A wave propagation algorithm for hyperbolic systems 
on curved manifolds.
\newblock {\em J. Comput. Phys.}, 199(2):631--662, 2004.

\bibitem{Seeley}
R.~T. Seeley.
\newblock Extension of {$C^{\infty }$} functions defined in a half space.
\newblock {\em Proc. Amer. Math. Soc.}, 15:625--626, 1964.

\end{thebibliography}
\end{document}